\documentclass[a4paper,11pt]{article}
\usepackage[a4paper, margin={1.05in, 1.05in}]{geometry}
\usepackage{amsmath,amsthm,amssymb,bm}
\usepackage{amssymb}
\usepackage{latexsym}
\usepackage[pdftex]{graphicx}
\usepackage{enumerate}
\usepackage{color}
\usepackage{soul}
\usepackage{caption}
\usepackage{multirow}

\title{Semi-random graph process}

\author{
	Omri Ben-Eliezer\thanks{Blavatnik School of Computer Science, Raymond and Beverly Sackler Faculty of Exact Sciences, Tel Aviv University, Tel Aviv, 6997801, Israel. Email: omrib@mail.tau.ac.il}
	\and Dan Hefetz\thanks{Department of Computer Science, Ariel University, Ariel 40700, Israel. Email: danhe@ariel.ac.il. Research supported by ISF grant 822/18.}
	\and Gal Kronenberg\thanks{School of Mathematical Sciences, Raymond and Beverly Sackler Faculty of Exact Sciences, Tel Aviv University, Tel Aviv, 6997801, Israel. Email: galkrone@mail.tau.ac.il.}
	\and Olaf Parczyk\thanks{Institut f\"ur Mathematik, Technische Universit\"at Ilmenau, 98684 Ilmenau, Germany. Email: olaf.parczyk@tu-ilmenau.de. Research supported by DFG grant PE 2299/1-1.}
	\and Clara Shikhelman\thanks{School of Mathematical Sciences, Raymond and Beverly Sackler Faculty of Exact Sciences, Tel Aviv University, Tel Aviv, 6997801, Israel. Email: clarashk@mail.tau.ac.il.}
	\and Milo\v{s} Stojakovi\'c\thanks{Department of Mathematics and Informatics, Faculty of Sciences, University of Novi Sad, Trg D.~Obradovi\'ca 4, 21000 Novi Sad, Serbia. Email: milos.stojakovic@dmi.uns.ac.rs. Partly supported by the Ministry of Education, Science, and Technological Development, Republic of Serbia, and the Provincial Secretariat for Higher Education and Scientific Research, Province of Vojvodina.}
}

\theoremstyle{plain}
\newtheorem{theorem}{Theorem}[section]
\newtheorem{lemma}[theorem]{Lemma}
\newtheorem{claim}[theorem]{Claim}
\newtheorem{proposition}[theorem]{Proposition}
\newtheorem{observation}[theorem]{Observation}
\newtheorem{corollary}[theorem]{Corollary}
\newtheorem{conjecture}[theorem]{Conjecture}

\newtheorem{question}[theorem]{Question}

\newenvironment{remark}{\paragraph{\textbf{Remark.}}}{\medskip}

\newcommand\cp{{\mathcal{P}}}

\newcommand\ch{{\mathcal{H}}}
\newcommand\cd{{\mathcal{D}}}

\newcommand\lf{\left}
\newcommand\rt{\right}
\newcommand{\Bin}{\ensuremath{\textrm{Bin}}}

\def\E{\mathbb{E}}

\newcommand\off{{\text{off}}}

\begin{document}
	\date{}
	\maketitle

\begin{abstract}
We introduce and study a novel semi-random multigraph process, described as follows. The process starts with an empty graph on $n$ vertices. In every round of the process, one vertex $v$ of the graph is picked uniformly at random and independently of all previous rounds. We then choose an additional vertex (according to a strategy of our choice) and connect it by an edge to $v$. 

For various natural monotone increasing graph properties $\mathcal{P}$, we prove tight upper and lower bounds on the minimum (extended over the set of all possible strategies) number of rounds required by the process to obtain, with high probability, a graph that satisfies $\mathcal{P}$. Along the way, we show that the process is general enough to approximate (using suitable strategies) several well-studied random graph models.
\end{abstract}

%%%%%%%%%%%%%%%%%%%%%%%%%%%%%%%%%%%%%%%%%%%%%%%%%%%%%%%%%%%%%%%%%%%%%%%%%%%%%%%
%%%%%%%%%%%%%%%%%%%%%%%%%%%%%%%%%%%%%%%%%%%%%%%%%%%%%%%%%%%%%%%%%%%%%%%%%%%%%%%

\section{Introduction}
In this paper we introduce and analyze a general semi-random multigraph process, arising from an interplay between a sequence of random choices on the one hand, and a strategy of our choice (that may also involve randomness) on the other.
The process is defined as follows. We start with an empty graph on the vertex set $[n]$. In each round,  \textit{Builder} is offered a vertex $v$, chosen uniformly at random (u.a.r.\ for brevity) with replacement from the set $[n]$, independently of all previous choices. Builder then irrevocably chooses an additional vertex  $u$ and adds the edge $uv$ to his (multi)graph, with the possibility of creating multiple edges and loops. 

The (possibly randomized) algorithm that Builder uses in order to add edges throughout this process is called the \textit{strategy} of Builder.
As a special case, we also show how the process can be used to approximate (using suitable strategies) some well-known random graph models such as the Erd\H{o}s-Renyi random graph model (see~\cite{ER}), the random multigraph model (see~\cite{HKK}), the $k$-out model (see~\cite{walkup1980matchings}), and the min-degree process (see~\cite{Wormald}).  

Given a positive integer $n$ and a monotone increasing graph property $\mathcal P$, we consider the one-player game in which Builder's goal is to build a multigraph with vertex set $[n]$ satisfying $\mathcal P$ as quickly as possible; we denote this game by $(\cp, n)$. The general problem discussed in this paper is to determine the typical number of rounds Builder needs in order to construct such a multigraph under optimal play. We mostly focus on the \emph{online version} of the game, where in each round Builder is presented with the next random vertex only after he chose a vertex in the previous round and added the corresponding edge to his graph, but also consider the \emph{offline version}, in which Builder is given the entire sequence of random vertex choices before the game starts.

\paragraph*{A formal treatment.}
Suppose that Builder follows some fixed strategy $\mathcal S$. Let $\mathcal S(n,m)$ denote the resulting multigraph if Builder follows $\mathcal S$ for $m$ rounds. That is, $\mathcal S(n,m)$ is the probability space of all multigraphs with vertex set $[n]$ and with $m$ edges, where each of these edges is chosen as follows. First a vertex $v \in [n]$ is chosen u.a.r., and then another vertex $u$ is chosen, according to $\mathcal S$, and the edge $uv$ is added to the graph. Sometimes, when $n$ and $\mathcal S$ are clear from the context, we use $G_m$ to denote Builder's multigraph after $m$ rounds. 
For the online version of the game $(\cp, n)$, for $0 \leq p \leq 1$ and for a strategy $\mathcal{S}$, we define $\tau_p(\mathcal{S})$ to be the smallest integer $m$ for which $G_m \sim \mathcal S(n,m)$ satisfies $\cp$ with probability at least $p$. If no such integer $m$ exists, then we define $\tau_p(\mathcal{S})$ to be $+\infty$. Furthermore, we define $\tau_{p}(\mathcal{P}, n)$ to be the minimum of $\tau_p(\mathcal{S})$, taken over all possible strategies $\mathcal{S}$ for $(\mathcal{P}, n)$. In other words, $\tau_p(\cp, n)$ can be seen as the smallest number of rounds Builder needs in order to build a multigraph which satisfies property $\cp$ with probability (at least) $p$, provided he adopts a best possible strategy for this purpose.
%in other words, $\tau_p(\cp, n)$ can be thought of as the hitting time of the strategy that is typically the fastest to win with probability at least $p$ among all possible strategies.

For the offline version of the game $(\cp, n)$ we define $\tau'_p(\mathcal{S})$ and $\tau'_p(\cp, n)$ in an analogous manner. Note that $\tau'_p(\cp, n) \leq \tau_p(\cp, n)$ holds for every $\cp$, $n$ and $p$.

For a given monotone increasing graph property $\cp$, our prime objective for the online version of the game $(\cp, n)$ is to obtain tight upper bounds on $\tau_{1-o(1)}(\cp, n)$ and tight lower bounds on $\tau_{o(1)}(\cp, n)$, where $o(1)$ is a positive function tending to zero as $n$ tends to infinity. With some abuse of notation, we say that these bounds are, respectively, upper and lower bounds on $\tau(\cp, n)$ which hold with high probability (w.h.p.~for brevity). Note that in order to prove that w.h.p. $\tau(\cp,n) \leq m$, it suffices to present a strategy $\mathcal S$ such that w.h.p. $G_m \sim \mathcal S(n,m)$ satisfies $\cp$. On the other hand, in order to prove that w.h.p. $\tau(\cp,n) \geq m$, one has to show that for any strategy $\mathcal S$, w.h.p. the graph $G_m \sim \mathcal S(n,m)$ does not satisfy $\cp$. Our prime objective for the offline version of the game is analogous, namely, to obtain tight upper and lower bounds on $\tau'(\cp, n)$ which hold with high probability.

In this paper we will establish such lower and upper bounds on $\tau(\cp, n)$ and on $\tau'(\cp, n)$ for several natural graph properties $\cp$. 
In the next three subsections we describe the main contributions of this paper: In Subsection \ref{subsec:connections} we explain how our model is connected to other random graph models; Subsection \ref{subsec:offline_results} is dedicated to our results on the offline version; and finally, in  Subsection \ref{subsec:online_results} we state tight upper and lower bounds for several properties of interest in the online version. 

\subsection{Connections to other random graph models}
\label{subsec:connections}
The first few results we obtain in this paper show that, in some sense, our process generalizes several well-known random graph models. Namely, given a certain model of random graphs $\mathcal G$, we prove that there exists an appropriate choice of a strategy ${\mathcal S}_{\mathcal G}$ for Builder such that ${\mathcal S}_{\mathcal G}(n,m)$ can be coupled to $\mathcal G$. Such results have independent interest, but will also enable us to use these well-studied models to draw conclusions about our semi-random graph process.   

%state that given It turns out that we can choose appropriate strategies $\mathcal S$ such that $\mathcal S(n,m)$ can be coupled to the several known random graph models. This will enable us to apply known results about those random graph models to draw conclusions about the semi-random graph process.

\medskip

\paragraph{\textbf{The random graph and multigraph models.}}
We first look at two classical random graph processes. The first process $\{G_m\}_{m=0}^{\binom{n}{2}}$ is defined as follows. Let $G_0$ be the empty graph with vertex set $[n] := \{1, \ldots, n\}$ and, for every $m \geq 0$, let $G_{m+1} = G_m \cup e$, where $e \in \binom{[n]}{2} \setminus E(G_m)$ is chosen u.a.r. It is well-known and easy to see that $G_m \sim G(n,m)$ for every $0 \leq m \leq \binom{n}{2}$, that is, this random graph process generates the Erd\H{o}s-Renyi random graph model~\cite{ER}. The second process we consider is the random multigraph process $\{M(n,m)\}_{m \geq 0}$ that was introduced in~\cite{HKK}. This process is similar to the first one except that, in each round, the edge we add is chosen u.a.r.~from $\binom{[n]}{2}$, allowing the graph to have multiple edges. For this reason, the process is also not limited in length. We prove that our semi-random graph process can be used to generate the second of these processes and to approximate the first.

\begin{proposition}\label{pro:Mnm}
	There exists a strategy $\mathcal S_{M}$ for Builder such that the probability space $\mathcal S_{M}(n,m)$ is the same as the probability space $M(n,m)$.

Furthermore,  there exists a strategy $\mathcal S_{G}$ for Builder such that if $m=o(n^2)$, then $H\sim G(n,m)$ and $H'\sim \mathcal S_{G}(n,(1+o(1))m)$ can be coupled
	in such a way that w.h.p.\ $H\subseteq H'$. Finally, if $m = o(n)$, then the strategy $\mathcal S_{G}$ is such that  $H\sim G(n,m)$ and $H'\sim \mathcal S_{G}(n,m)$ can be coupled in such a way that w.h.p.\ $H= H'$.
\end{proposition}

The following two results are immediate corollaries of Proposition~\ref{pro:Mnm}.
\begin{corollary}\label{cor:Gnm}
Let $m_{\mathcal P}$ be a positive integer for which w.h.p.\ $H \sim G(n,m_{\mathcal P})$ satisfies the monotone increasing graph property $\mathcal P$. If $m_{\mathcal P} = o \left(n^2 \right)$, then w.h.p.\ $\tau(\mathcal P,n) \leq (1+o(1)) m_{\mathcal P}$.
\end{corollary}

\begin{corollary}\label{cor:Mnm}
Let $m_{\mathcal P}$ be a positive integer for which w.h.p.\ $M \sim M(n, m_{\mathcal P})$ satisfies the monotone increasing graph property $\mathcal P$. Then w.h.p.\ $\tau(\mathcal P,n) \leq m_{\mathcal P}$.
\end{corollary}

\paragraph{\textbf{The $k$-out model.}} 

%We move on to the random graph model $\mathcal G_{k\text{-out}}(G)$, where $G$ is a graph with minimum degree $\delta(G)$ and $k\leq \delta(G)$ is a positive integer. 
Given a graph $G$ with minimum degree $\delta$ and a positive integer $k \leq \delta$, let $\mathcal G_{k\text{-out}}(G)$ denote the probability space of subgraphs $H$ of $G$ obtained via the following procedure: each vertex $v \in V(G)$ chooses $k$ out-neighbors uniformly at random among its neighbors in $G$ to create a digraph $D$; then, $H$ is obtained by ignoring orientations in $D$ and replacing multiple edges with single edges. We abbreviate $\mathcal G_{k\text{-out}}(K_n)$ under $\mathcal G_{k\text{-out}}(n)$. This model first appeared in ``The Scottish Book"~\cite{Scottish} and was also introduced by Walkup in 1980~\cite{walkup1980matchings}, where he proved that for every sufficiently large integer $n$, a graph $H \sim \mathcal G_{2\text{-out}}(K_{n,n})$ typically admits a perfect matching.  

The following result asserts that a typical $G \sim \mathcal G_{k\text{-out}}(n)$ can be approximated using our semi-random graph process (another related result will be discussed in Subsection~\ref{subsec::kout}).

%Given $G$ and $k$, we define $\mathcal G_{k\text{-out}}(G)$ to be the distribution of subgraphs $H$ of $G$ obtained by the following procedure:
%each vertex $v\in V(G)$ chooses $k$ out-neighbors u.a.r.~among its neighbors in $G$ to create a digraph $D$; then, $H$ is obtained by ignoring orientations in $D$ and replacing multiple edges with one edge. Where $G=K_n$ we simply write $\mathcal G_{k\text{-out}}(n)$. This model first appeared in ``The Scottish Book" \cite{Scottish} and  was also introduced by Walkup in 1980~\cite{walkup1980matchings}, where
%he proved that for every sufficiently large integer $n$, a graph $H\sim\mathcal G_{2\text{-out}}(K_{\lfloor\frac n2\rfloor,\lfloor\frac n2\rfloor})$ typically contains a perfect matching.

%In the following statement we show that a typical $\mathcal G_{k\text{-out}}(n)$ can be generated using the semi-random graph process.
\begin{proposition}\label{pro:Kout}
Fix a positive integer $k$. There exists a strategy $\mathcal S_{out}$ for Builder such that  $H \sim \mathcal G_{k\text{-out}}(n)$ and $G \sim \mathcal S_{out}(n, k n + o(n))$ can be coupled in such a way that w.h.p.\ $H \subseteq G$.
\end{proposition}

%\begin{proposition}\label{pro:Kout}
%	Let $k$ be a positive integer. Then there exists a strategy $\mathcal S_{out}$ for Builder such that  $H\sim \mathcal G_{k\text{-out}}(n)$ and $G\sim \mathcal S_{out}(n,(k + o(1))n)$ can be coupled in such a way that w.h.p.\ $H\subseteq G$.
%\end{proposition}

The following result is an immediate corollary of Proposition~\ref{pro:Kout}.
\begin{corollary} \label{cor:Kout}
	Let $k_{\mathcal P}$ be a positive integer for which w.h.p.\ $H \sim \mathcal G_{k_{\mathcal P}\text{-out}}(n)$ satisfies the monotone increasing graph property $\mathcal P$. Then w.h.p.\ $\tau(\mathcal P,n) \leq (k_{\mathcal P} + o(1)) n$.
\end{corollary}

Corollary~\ref{cor:Kout} has several consequences. For example, it implies an upper bound on the duration of the online Hamiltonicity game and the perfect matching game. This is further discussed in Subsection~\ref{subsec::kout} and Section~\ref{sec::openprob}.

%A similar argument shows that we can also approximate $\mathcal G_{(1+e^{-1})\text{-out}}(K_{n/2,n/2})$ using our semi-random graph process.

%\begin{proposition}\label{pro:KoutBip}
%	There exists a strategy $\mathcal S'_{out}$ for Builder such that  $H\sim \mathcal G_{(1+e^{-1})\text{-out}}(K_{\lfloor\frac n2\rfloor,\lfloor\frac n2\rfloor})$ and $G\sim \mathcal S'_{out}(n,(1+e^{-1}+o(1))n)$ can be coupled in such a way that w.h.p.\ $H\subseteq G$.
%\end{proposition}

%Using Proposition~\ref{pro:KoutBip} we get the following corollary, which will be useful for studying $(\cpm, n)$ game.

%\begin{corollary}\label{cor:KoutBip}
%	Assume that w.h.p.\ $H\sim\mathcal G_{(1+e^{-1})\text{-out}}(K_{\lfloor\frac n2\rfloor,\lfloor\frac n2\rfloor})$ satisfies a monotone increasing graph property $\mathcal P$. Then w.h.p.\ $\tau(\mathcal P,n)\leq (1+e^{-1}+o(1))n$.
%\end{corollary}

\medskip

\paragraph{\textbf{A min-degree process.}} The graph process $\{G_{\min}(n,m)\}_{m \geq 0}$, introduced 
by Wormald in 1995~\cite{Wormald}, is defined as follows. Let $G_{\min}(n, 0)$ be the empty graph with vertex set $[n]$ and, for every $m \geq 0$, let $G_{\min}(n, m+1)$ be obtained from $G_{\min}(n,m)$ by first choosing a vertex $u$ of minimum degree in $G_{\min}(n,m)$ u.a.r., and then connecting it by a new edge to a vertex $v \in [n] \setminus \{u\}$ chosen u.a.r.~among all vertices which are not connected to $u$ by an edge of $G_{\min}(n,m)$. For more about this process (and other related processes), see, e.g.,~\cite{KKRL, KS, Wormald}. 

We show that this process can be approximated using the semi-random graph process.

\begin{proposition}\label{pro:Gmin}
	If $m = o \left(n^2 \right)$, then there exists a strategy $\mathcal S_{min}$ for Builder such that  $H \sim G_{\min}(n,m)$ and $G \sim \mathcal S_{min}(n, (1+o(1))m)$ can be coupled
	in such a way that w.h.p.\ $H \subseteq G$.
\end{proposition}

The next result is an immediate corollary of Proposition~\ref{pro:Gmin}.

\begin{corollary}\label{cor:Gmin}
	Let $m_{\mathcal P}$ be a positive integer for which w.h.p.\ $H \sim G_{min}(n,m_{\mathcal P})$ satisfies the monotone increasing graph property $\mathcal P$. If $m_{\mathcal P} = o \left(n^2 \right)$, then w.h.p.\ $\tau(\mathcal P, n) \leq (1+o(1)) m_{\mathcal P}$.
\end{corollary}

In order to prove Proposition~\ref{pro:Gmin} we use two other min-degree processes (which we also approximate using the semi-random graph process), presented in Subsection~\ref{subsec:min_deg_process}.

%\paragraph{\textbf{Preferential Attachment}} 

\medskip

Corollaries~\ref{cor:Gnm}, \ref{cor:Mnm}, \ref{cor:Kout} and~\ref{cor:Gmin} show that building a graph which satisfies some monotone increasing graph property $\mathcal P$ in our semi-random graph process is at least as fast as building such a graph in various other well-known graph processes. As we will see in the next two subsections (see Theorems~\ref{thm:OffFix} and~\ref{th::fixedGraphUpperBound}), for some properties the semi-random graph process is much faster than all of these other processes.

\subsection{Offline games}
\label{subsec:offline_results}
Here we state the upper and lower bounds on $\tau'(\mathcal{P}, n)$ that we obtain for various natural graph properties $\mathcal{P}$. Given a fixed graph $H$, let $\mathcal{P}_H$ denote the property of containing $H$ as a subgraph. Our next result determines the order of magnitude of $\tau'(\mathcal{P}_H, n)$ for every fixed graph $H$.      

\begin{theorem}\label{thm:OffFix}
	Let $H$ be a fixed graph and let $r = \min \{\Delta^+(D) : D \textrm{ is an orientation of } H\}$. Let $f, g : \mathbb{N} \to \mathbb{R}$ be functions such that $f(n)$ tends to zero arbitrarily slowly as $n$ tends to infinity and $g(n)$ tends to infinity arbitrarily slowly as $n$ tends to infinity. Then w.h.p. $f(n) \cdot n^{(r-1)/r} \leq \tau'(\mathcal{P}_H, n) \leq g(n) \cdot n^{(r-1)/r}$. 	
\end{theorem}

Given a positive integer $k$, let ${\cd}_k$ denote the property of having minimum degree at least $k$. Our next result determines $\tau'({\cd}_k, n)$ asymptotically for every fixed positive integer $k$.

\begin{theorem} \label{th::MinDegkOffline}
	Let $k$ be a positive integer and let $\alpha_k$ be the unique positive root of $f_k(x) := \sum_{i=0}^{k-1} \left(k-i \right) \frac{x^i}{i!} - x e^x$. Then w.h.p. $\tau'({\cd}_k, n) = (\alpha_k + o(1)) n$.
\end{theorem}

\begin{remark}
	It is not hard to see that, for every $k$, indeed $f_k(x)$ has exactly one positive root. Moreover, $\lim_{k \rightarrow \infty} \alpha_k = k/2$, that is, in the offline version, when $k$ is sufficiently large, Builder has a strategy to build a graph with minimum degree $k$ in $kn/2 + o(n)$ rounds.
\end{remark}

\subsection{Online games}
\label{subsec:online_results}

In this subsection we state the upper and lower bounds on $\tau(\mathcal{P}, n)$ that we obtain for various natural graph properties $\mathcal{P}$.

\begin{theorem} \label{th::fixedGraphUpperBound}
%	Let $H$ be a fixed graph and let $d$ be its degeneracy. Then w.h.p. $\tau({\mathcal P}_H, n) = O \left(n^{(d-1)/d} \right)$.
%	
	Let $H$ be a
	fixed graph and let $d$ be its degeneracy. Then w.h.p. $\tau({\mathcal
		P}_H, n) \leq g(n) \cdot n^{(d-1)/d}$, where $g : \mathbb{N} \to
	\mathbb{R}$ is a function which tends to infinity arbitrarily slowly
	as $n$ tends to infinity
%	
%	Let $H$ be a fixed graph. Then there exists a constant $C>0$ such that w.h.p.\ Builder can win the $(\mathcal{P}_H, n)$ game within $C n^{{(d-1)}/{d}}$ rounds, where $d$ is the degeneracy of $H$.
\end{theorem}
%
%The strategy that Builder uses is a recursive one. For example, in cliques it amounts to the following: Building a clique with two vertices is trivial (and requires exactly one round). Suppose that a copy of $K_{t-1}$ was already constructed, and its vertex set is $V_{t-1}$.
%In order to obtain a copy of $K_t$ out of the above copy of $K_{t-1}$, in each round of the process we choose a vertex $u \in V_{t-1}$ that is currently not connected to the offered vertex $v$, and add the edge between $u$ and $v$. Thus, as soon as a vertex $v \notin V_{t-1}$ is offered $t-1$ times, it becomes connected to all vertices of the copy; adding $v$ to the copy extends it into a copy of $K_t$, as desired. 

It seems plausible that the upper bound stated in Theorem~\ref{th::fixedGraphUpperBound} is of the correct order of magnitude for every fixed graph $H$. At the moment, we can only prove this for cliques.

\begin{theorem} \label{thm:lowerCliqueOnLine}
%	For a positive integer $d$, let $\mathcal P_d$ denote the graph property of containing a copy of $K_d$ as a subgraph. Then w.h.p. $\tau(\mathcal P_d, n) = \Omega \left(n^{(d-2)/(d-1)} \right)$ for every $d \geq 2$.
%	
	For a positive integer $d$, let ${\mathcal P}_d$ denote the graph
	property of containing a copy of $K_d$ as a subgraph. Then w.h.p.
	$\tau({\mathcal P}_d, n) \geq f(n) \cdot n^{(d-2)/(d-1)}$ for every $d
	\geq 2$, where $f : \mathbb{N} \to \mathbb{R}$ is a function which
	tends to zero arbitrarily slowly as $n$ tends to infinity
	%	
	%	Let $H=K_{d+1}$ be the complete graph on $d+1$ vertices. Define $\mathcal{P}_H$ to be the property of containing a copy of $H$ as a subgraph.  Let $m=o(n^{{(d-1)}/{d}})$, then w.h.p.\ Builder can not win the $(\mathcal{P}_H,n)$ game after $m$ rounds.
\end{theorem}

A consequence of Theorems~\ref{thm:lowerCliqueOnLine} and~\ref{thm:OffFix} is that, perhaps not surprisingly, $\tau(\mathcal P_d, n) \gg \tau'(\mathcal P_d, n)$ for every $d \geq 3$.

\medskip

%Let $\mathcal{D}_k$ denote the (multi)graph property of having minimum degree\footnote{As we shall see, for constant $k$, it does not matter whether we include loops and multiple edges when calculating the degree of each vertex; Theorem \ref{th::OnMinDegk} holds for all relevant cases.} at least $k$. 
Our next result determines $\tau({\cd}_k, n)$ asymptotically for every fixed positive integer $k$.

\begin{theorem} \label{th::OnMinDegk}
	For every positive integer $k$, there exists a positive real number $h_k$ such that w.h.p. $\tau({\cd}_k, n) = (h_k + o(1)) n$. 
\end{theorem}
%The strategy used to prove the upper bound in Theorem \ref{th::OnMinDegk} is $\mathcal S_{\min}$ used in Proposition~\ref{pro:Gmin}. We prove a matching lower bound by showing that $\mathcal S_{\min}$ is nearly optimal in some precise sense that we will define in Subsection \ref{subsec:online_min_deg_k}.  

Note that the real numbers $h_k$ appearing in the statement of Theorem~\ref{th::OnMinDegk} can be (and were) calculated using Wormald's differential equations method~\cite{Wormald, Wormald2}. We discuss this in greater detail in Subsection~\ref{subsec:online_min_deg_k}.

\medskip

Given a positive integer $k$, let ${\mathcal C}_k$ denote the property of being $k$-vertex-connected. Clearly, $\tau({\mathcal C}_k, n) \geq \tau({\mathcal D}_k, n)$ holds for every $k$. Our next result shows that, like in several other graph models, ${\mathcal C}_k$ and ${\mathcal D}_k$ occur roughly at the same time in our semi-random graph process.
% as we shall later show, the suitable strategy for this case is a very slight modification of $\mathcal S_{\min}$.  

\begin{theorem} \label{thm:OnkCon}
	Let $k \geq 3$ be a fixed integer. Then w.h.p. $\tau({\mathcal C}_k, n) = (h_k + o(1)) n$.
\end{theorem}	

Note that, trivially, $\tau({\mathcal C}_1, n) = n-1$. On the other hand, the best bounds we currently have for $\tau({\mathcal C}_2, n)$ stem from the fact that $(h_2 + o(1))n = \tau({\mathcal D}_2, n) \leq \tau({\mathcal C}_2, n) \leq \tau({\mathcal C}_3, n) = (h_3 + o(1))n$. 

\medskip

We summarize the results of the last two subsections in the following table.
\begin{table}[h!]
	\centering
	\begin{tabular}{||c|| c c | c c||} 
		\hline
		Property & Online Bounds & & Offline Bounds & \\ [0.5ex] 
		\hline\hline
		\multirow{2}{*}{$\mathcal{P}_H$} &  \multirow{2}{*}{ $\leq g(n) \cdot n^{(d-1)/d}$} & \multirow{2}{*}{(\ref{th::fixedGraphUpperBound})}  &   $\geq f(n) \cdot n^{(r-1)/r}$ & \multirow{2}{*}{(\ref{thm:OffFix})}
		\\ & & & $\leq g(n) \cdot n^{(r-1)/r}$ &
		\\
		\hline
		${\mathcal P}_d$ &  $\geq f(n) \cdot n^{(d-2)/(d-1)}$ & (\ref{thm:lowerCliqueOnLine}) & &  \\
		\hline
		${\cd}_k$ & $= (h_k + o(1)) n$ & (\ref{th::OnMinDegk}) & $\alpha_k + o(1)$ & (\ref{th::MinDegkOffline}) \\
		\hline
		${\mathcal C}_k$ & = $(h_k + o(1)) n$ &(\ref{thm:OnkCon}) &  &  \\
		\hline
		\multicolumn{5}{c}{}\\
		 \multicolumn{5}{c}{$g$ tends to infinity arbitrarily slowly as $n$ tends to infinity}\\
		 \multicolumn{5}{c}{$f$ tends to zero arbitrarily slowly as $n$ tends to infinity}\\
		 \multicolumn{5}{c}{ $h_k$ is a known constant }\\
		 \multicolumn{5}{c}{$\alpha_k$ is the unique positive root of $f_k(x) := \sum_{i=0}^{k-1} \left(k-i \right) \frac{x^i}{i!} - x e^x$}
		 
		 %\multirow{2}{c}{\multicolumn{5}{c}{ aaaa}}
	\end{tabular}
	
	%\caption*{ $g$ tends to infinity arbitrarily slowly as $n$ tends to infinity \\ $f$ tends to zero arbitrarily slowly as $n$ tends to infinity \\ $h_k$ is an integer \\ $\alpha_k$ is the unique positive root of $f_k(x) := \sum_{i=0}^{k-1} \left(k-i \right) \frac{x^i}{i!} - x e^x$}
	%\label{table:1}
\end{table}

\subsection{Organization}

The rest of the paper is structured as follows. In Subsection~\ref{subsec::notation} we list some, mostly standard, notation that will be used throughout the paper. In Subsections~\ref{subsec::prob_tools} and~\ref{sec::BallsBins} we collect various useful probabilistic tools that will be used in later sections. In Section~\ref{sec:OtherModels} we show how to use our semi-random process to approximate various random graph models; in particular, we prove Propositions~\ref{pro:Mnm}, \ref{pro:Kout},  and~\ref{pro:Gmin}. In Section~\ref{sec:off} we study the offline version of the process; in particular, we prove Theorems~\ref{thm:OffFix} and~\ref{th::MinDegkOffline}. In Section~\ref{sec:on} we study the online version of the process; in particular, we prove Theorems~\ref{th::fixedGraphUpperBound}, \ref{thm:lowerCliqueOnLine}, \ref{th::OnMinDegk}, and~\ref{thm:OnkCon}. Finally, Section~\ref{sec::openprob} contains a discussion of our semi-random graph process, including open problems and suggestions for future research.

%%%%%%%%%%%%%%%%%%%%%%%%%%%%%%%%%%%%%%%%%%%%%%%%%%%%%%%%%%%%%%%%%%%%%%%%%%%%%%%

\section{Preliminaries} \label{sec::prelims}

\subsection{Notation} \label{subsec::notation}
Let $\mathbb{N}$ denote the set of all non-negative integers and let $\mathbb{R}$ denote the set of all real numbers. For a positive integer $n$ let $[n] = \{1, \ldots, n\}$.

In a (multi)graph $G=(V,E)$, $V$ is the set of vertices and $E$ is a
multiset of elements from $\binom {V}2$.  
Let $d_G(v)=|\{e\in E\ :\  v\in e \}|$
denote the degree of $v$ in $G$, that is, the number of edges incident to $v$ including multiplicities.
Let $\Delta(G)$ denote the maximum degree of $G$ and $\delta(G)$ the minimum degree of $G$.
For a vertex $v\in V$, let $N_G(v)=\{u\in V\setminus\{v\}\ :\  uv\in E \}$ denote the set of distinct neighbors of the vertex $v$. Note that for every $v\in V$, $|N_G(v)|\leq d_G(v)$.

For a directed graph $D$ and a vertex $v \in V(D)$, let $d^+_D(v)$ denote the out-degree of $v$ in $D$ and let $\Delta^+(D) = \max \{d_D^+(v) : v \in V(D)\}$ denote the maximum out-degree of $D$. Often, if there is no risk of confusion, we abbreviate $d_D^+(v)$ under $d^+(v)$. 
A graph $G$ is called $k$-degenerate if every subgraph of $G$ has a vertex of degree at most $k$. Equivalently, a graph $G$ is $k$-degenerate if there exists an ordering $(v_1, \ldots, v_n)$ of its vertices  such that $v_i$ has at most $k$ neighbors in the set $\{v_1, \ldots, v_{i-1}\}$ for every $2 \leq i \leq n$. The \textit{degeneracy} of a graph $G$ is the smallest value of $k$ for which $G$ is $k$-degenerate.   

For a monotone increasing property $\mathcal P$ and a vertex $v \in V(K_n)$, at any point during the game $(\mathcal P, n)$, let $\off(v)$ denote the number of times $v$ was offered up to that point.

\subsection{Probabilistic tools}
\label{subsec::prob_tools}
The following well-known bounds on the tails of the binomial distribution, due to Chernoff (see, e.g.,
\cite{AloSpe2008}, \cite{JLR}), will be used extensively.
%
%One of the most famous bounds on the tails of the binomial distribution, which we use extensively in this paper, is due to Chernoff (see, e.g.,
%\cite{AloSpe2008}, \cite{JLR}).
%
%\begin{lemma}\label{Che}
%	Let $X\sim \Bin(n,p)$
%%	or $X\sim HG(N,n,m)$
%	and $\mu=\mathbb{E}(X)$, then
%	\begin{enumerate}
%		\item $\Pr\left(X<(1-a)\mu\right)<\exp\left(-\frac{a^2\mu}{2}\right)$ for every $a>0.$
%		\item $\Pr\left(X>(1+a)\mu\right)<\exp\left(-\frac{a^2\mu}{3}\right)$ for every $0 < a < 1.$
%	\end{enumerate}
%\end{lemma}
%
%A stronger version of Chernoff's bound is the following (see, e.g., \cite{JLR}, Theorem 2.1).

\begin{lemma}\label{Che}
	Let $X\sim \Bin(n,p)$,	$\mu=\mathbb{E}(X)$ and $a\geq 0$, then
	\begin{enumerate}
		\item $\Pr\left[X\leq \mu-a\right]\leq  \exp\left(-\frac{a^2}{2\mu}\right)$;
		\item $\Pr\left[X\geq \mu+a\right]\leq \exp\left(-\frac{a^2}{2(\mu+\frac a3)}\right)$.
	\end{enumerate}
\end{lemma}

The following is a well-known concentration inequality due to Azuma~\cite{Azuma}.

\begin{theorem} \cite[Theorem 2.27]{JLR} \label{th::Azuma}
	Suppose that $Z_1, \ldots, Z_m$ are independent random variables taking their values in the set $[n]$. Suppose further that $X = f(Z_1, \ldots, Z_m)$, where $f : [n]^m \to {\mathbb R}$ is a function such that there exist constants $c_1, \ldots, c_m$ for which the following condition holds:
	\begin{description}
		\item [(a)] If $z, z' \in [n]^m$ differ only in the $k$th coordinate, then $|f(z') - f(z)| \leq c_k$.
	\end{description}
	Then for every $t \geq 0$ we have
	$$
	\Pr (|X - \E(X)| \geq t) \leq 2 \exp \left\{- \frac{t^2}{2 \sum_{k=1}^m c_k^2} \right\} \,.
	$$
\end{theorem}

The following is a simplified version of a concentration inequality due to Talagrand~\cite{Talagrand}.

\begin{theorem}  \cite[page 81]{MR} \label{th::Talagrand}
	Let $X$ be a non-negative random variable, not identically $0$, which is determined by $n$ independent trials $T_1, \ldots, T_n$, and satisfying the following for some $c, r > 0$:
	\begin{description}
		\item [(a)] changing the outcome of any one trial can affect $X$ by at most $c$, and
		\item [(b)] for any $s$, if $X \geq s$, then there is a set of at most $r s$ trials whose outcomes certify that $X \geq s$,
	\end{description}
	then for any $0 \leq t \leq \mathbb{E}(X)$
	$$
	\Pr \left(|X - \mathbb{E}(X)| > t + 60 c \sqrt{r \mathbb{E}(X)} \right) \leq 4 \exp \left\{ -\frac{t^2}{8 c^2 r \mathbb{E}(X)} \right\} \,.
	$$
\end{theorem}

\subsection{Balls Into Bins} \label{sec::BallsBins}

Consider $m$ balls, placed into $n$ bins labeled $1, 2, \ldots, n$, where for each ball, we choose a bin u.a.r.\ and independently from all previous choices. For every $1 \leq i \leq m$, let $Z_i$ denote the bin chosen for ball $i$. Note that $Z_1, \ldots, Z_m$ are independent random variables. For every non-negative integer $k$, let $X_k^m = X_k^m(n)$ be the random variable counting the number of bins containing exactly $k$ balls and let $f(n, m, k) = \E(X_k^m)$. It is evident that
\begin{equation} \label{eq::expectationXkmn}
	f(n,m,k) = n \binom{m}{k} \lf( \frac{1}{n} \rt)^{k} \lf( 1 - \frac{1}{n} \rt)^{m-k}.
\end{equation}

If $k = o(\min \{n, \sqrt{m}\})$, then~\eqref{eq::expectationXkmn} takes the following simpler form:
%If $k$ is constant {\color{red}Is it true also for $k=o(\min\{n,\sqrt m\})$? because we use it...}, $n = \omega(1)$ and $m = \omega(1)$, then~\eqref{eq::expectationXkmn} takes the following simpler form:
\begin{equation} \label{e:exactly-k}
	f(n,m,k) = (1 + o(1)) \frac{e^{- m/n}}{k!} \cdot \frac{m^k}{n^{k-1}}.
\end{equation}

The following bound on the maximum number of balls in any single bin is an immediate consequence of~\eqref{e:exactly-k}. %{\color{red}how?}
\begin{corollary} \label{cor::maxLoad}
	If $m = O(n)$, then w.h.p. $\max \{k : X_k^m(n) > 0\} \leq \log n$.
\end{corollary}

\begin{remark}
	Much more accurate bounds on $\max \{k : X_k^m(n) > 0\}$ are known for a wider range of values of $m$ (see, e.g.,~\cite{RS}). However, Corollary~\ref{cor::maxLoad} will suffice for the purposes of this paper.
\end{remark}

We next prove that $X_k^m$ is concentrated around its mean.

\begin{lemma} \label{lem::Balls}
	For every $t \geq 0$
	$$
	\Pr(|X_k^m - f(n,m,k)| \geq t) \leq 2 \exp \{- t^2/(8 m)\}.
	$$
\end{lemma}

\begin{proof}
	It is evident that $X_k^m = f(Z_1, \ldots, Z_m)$ for some function $f : [n]^m \to {\mathbb R}$. Moreover, since moving one ball from one bin to another can change $X_k^m$ by at most $2$, it follows that Property (a) from Theorem~\ref{th::Azuma} holds with $c_1 = \ldots = c_m = 2$. Applying Theorem~\ref{th::Azuma} to $X_k^m$, we conclude that
	$$
	\Pr(|X_k^m - f(n,m,k)| \geq t) \leq 2 \exp \{- t^2/(8 m)\}.
	$$
\end{proof}

Lemma~\ref{lem::Balls} can be used to show that $X_k^m$ is concentrated around its mean whenever $f(n,m,k) = \omega(\sqrt{m})$. For smaller values of $f(n,m,k)$, we use the following lemma.

\begin{lemma} \label{lem::BallsTalagrand}
	If $\sum_{i = k+1}^m X_i^m = 0$, then
	$$
	\Pr \left(|X_k^m - f(n,m,k)| > t + 120 \sqrt{k f(n,m,k)} \right) \leq 4 \exp \left\{ -\frac{t^2}{32 k f(n,m,k)} \right\},
	$$
	for every $0 \leq t \leq f(n,m,k)$.
\end{lemma}

\begin{proof}
	It is evident that $X_k^m$ is a non-negative random variable which not identically $0$, and that $X_k^m$ is determined by $m$ independent trials. Moreover, since moving one ball from one bin to another can change $X_k^m$ by at most $2$, it follows that Property (a) from Theorem~\ref{th::Talagrand} holds with $c = 2$. Now, if $X_k^m \geq s$, then there are $s$ bins, each containing exactly $k$ balls. Since $\sum_{i = k+1}^m X_i^m = 0$ by assumption, it follows that there are $k s$ balls which certify that $X_k^m \geq s$. Therefore, Property (b) from Theorem~\ref{th::Talagrand} holds with $r = k$. We can thus apply Theorem~\ref{th::Talagrand} to deduce that
	$$
	\Pr \left(|X_k^m - f(n,m,k)| > t + 120 \sqrt{k f(n,m,k)} \right) \leq 4 \exp \left\{ -\frac{t^2}{32 k f(n,m,k)} \right\}
	$$
	holds for every $0 \leq t \leq f(n,m,k)$ as claimed.
\end{proof}

%%%%%%%%%%%%%%%%%%%%%%%%%%

\section{Connections to other random graph models} \label{sec:OtherModels}

\subsection{Random graph model and random multigraph model}

\begin{proof}[Proof of Proposition~\ref{pro:Mnm}]
	
	Builder's strategies for approximating both processes are similar.
	
	\textbf{Strategy $\mathcal S_M$:} When offered some vertex $v \in [n]$, he chooses a vertex $u$ u.a.r.~from $[n] \setminus \{v\}$. Then, Builder adds the edge $uv$ to his graph.
	
	\textbf{Strategy $\mathcal S_G$:}
	When offered some vertex $v \in [n]$, he chooses a vertex $u$ u.a.r.~from $[n] \setminus \{v\}$. Then, Builder adds $uv$. If $uv$ was added before, then he considers this round as a failure (which will not be part of the graph he aims to build).
	
	It is easy to see that, following $\mathcal S_M$, the probability of any edge being chosen in any round is $2 \cdot \frac{1}{n} \cdot \frac{1}{(n-1)} = \binom{n}{2}^{-1}$ which is precisely the probability of this edge to be chosen in $\{M(n,m)\}_{m \geq 0}$. That is, these two processes are identical.
	
	As for the classical random graph process, the two processes $G(n,m)$ and  $\mathcal S_G(n,m)$ begin to differ as soon as the first multiple edge is claimed. We begin by showing that w.h.p.  this does not happen if $m = o(n)$. Indeed, fix $m = o(n)$ and assume we have run $\mathcal S_G(n,k)$  for some $0 \leq k < m$. The probability that the edge chosen in round $k+1$ already exists in $G_k\sim \mathcal S_G(n,k)$ is at most $k \binom{n}{2}^{-1}$. Hence, the probability that $G_m$ contains multiple edges is at most $\sum_{i=0}^{m-1} i \binom{n}{2}^{-1} = \binom{m}{2} \binom{n}{2}^{-1} = o(1)$.
	
	As $m$ gets larger, some failures are expected, that is, w.h.p. Builder's (multi)graph will contain multiple edges. However, their number will be negligible assuming that $m = o(n^2)$. Indeed, by the above calculation, the expected number of failures after $m$ rounds of running our process according to $\mathcal S_G$ is at most $\sum_{i=0}^{m-1} i \binom{n}{2}^{-1} = \binom{m}{2} \binom{n}{2}^{-1} = o(m)$. Hence, it follows by Markov's inequality that w.h.p.\ there are $o(m)$ failures.
\end{proof}

\subsection{The $k$-out model} \label{subsec::kout}

In this subsection we prove Proposition~\ref{pro:Kout} by describing a strategy $\mathcal S$ for which $\mathcal S(n,m)$ can be coupled to the well-studied random graph model $\mathcal G_{k\text{-out}}(G)$ (where $G$ is a graph on $n$ vertices, $k \leq \delta(G)$ is a positive integer, and $m$ roughly equals $k n$). 
The strategy $\mathcal S_{out}$ is roughly described as follows. Suppose that the set of vertices is $[n]$. For every $1 \leq t \leq n$, we consider $k$ consecutive rounds of the process, and in all such rounds we connect the offered vertex to $t$. If at least one of these rounds adds a loop or an edge that already exists, we run $r$ extra rounds (for a suitable $r$), and connect all vertices offered in these rounds to $t$ as well. As it turns out, for every fixed $k$, the total number of extra rounds required is w.h.p. $o(n)$.

%\begin{proposition}\label{pro:Kout}
%	Let $k$ be a positive integer. Then there exists a strategy $\mathcal S_{out}$ for Builder such that  $H \sim \mathcal G_{k\text{-out}}(n)$ and $G \sim \mathcal S_{out}(n, k n + o(n))$ can be coupled in such a way that w.h.p.\ $H \subseteq G$. 
%\end{proposition}

\begin{proof} [Proof of Proposition~\ref{pro:Kout}]
	For positive integers $k$ and $r$, let $\mathcal G'_{(k,r)\text{-out}}(n)$ be the probability space of multigraphs $H$ obtained via the following procedure:
	each vertex $i \in [n]$ chooses $k$ out-neighbors, one by one, u.a.r.\ with replacement among all the vertices in $[n]$ (including $i$). If, during these $k$ choices, $i$ chose itself or the same vertex more than once, then $i$ chooses $r$ additional out-neighbors, using the same procedure, to create a digraph $D$. Finally, $H$ is obtained from $D$ by ignoring orientations (note that $H$ might contain loops and multiple edges).
	
	Given positive integers $k$ and $r$, let $\mathcal S_{out} = \mathcal S_{out}(k,r)$ be the following strategy of Builder.  
	
	\medskip	
	
	\noindent	\textbf{Strategy $\mathcal S_{out}$:} For every $i \geq 1$, let $u_i$ denote the vertex Builder is offered in the $i$th round. In the first $k$ rounds, Builder claims the edges $1 u_1, \ldots, 1 u_k$. If $1 \in \{u_1, \ldots, u_k\}$ or $u_i = u_j$ for some $1 \leq i < j \leq k$, then in the next $r$ rounds Builder claims the edges $1 u_{k+1}, \ldots, 1 u_{k+r}$. Builder now chooses the ``out-neighbors'' of $t$ for every $2 \leq t \leq n$ similarly. Namely, assume that for some $2 \leq t \leq n$ Builder has already chosen the ``out-neighbors'' of $j$ for every $1 \leq j \leq t-1$ (using the process described above for $t=2$) and this took him $m_{t-1}$ rounds. In rounds $m_{t-1} + 1, \ldots, m_{t-1} + k$, Builder claims the edges $t u_{m_{t-1} + 1}, \ldots, t u_{m_{t-1} + k}$. If $t \in \{u_{m_{t-1} + 1}, \ldots, u_{m_{t-1} + k}\}$ or $u_i = u_j$ for some $m_{t-1} + 1 \leq i < j \leq m_{t-1} + k$, then in the next $r$ rounds Builder claims the edges $t u_{m_{t-1} + k + 1}, \ldots, t u_{m_{t-1} + k + r}$.  
	
	\medskip	
	
	%	Order the vertices by $v_1,\dots,v_n$. In the $i$th block, $A_i$, connect the next $k$ vertices that are offered to the vertex $v_i$. Then, if the first $k$ elements in $A_i$ contains either $v_i$ or two identical vertices, then
	%	connect also the next $r$ vertices that are offered to the vertex $v_i$ and move to the next block. 
	%	Otherwise, move to the next block. Denote by $n'$ the number of blocks with $k+r$ elements.
	
	For every $1 \leq i \leq n$, let $A_i$ denote the multiset of ``out-neighbors'' Builder chooses for $i$ when playing according to $\mathcal S_{out}$; we refer to $A_i$ as the \emph{$i$th block}. Observe that for every $1 \leq i \leq n$, we have $|A_i| = k$ (in which case we will say that $A_i$ is \emph{small}) or $|A_i| = k+r$ (in which case we will say that $A_i$ is \emph{big}). Let $n'$ denote the number of big blocks.   
	
	\begin{observation}\label{obs:kout}
		The probability space $\mathcal S_{out}(n, k n + r n')$ is the same as the probability space $\mathcal G'_{(k,r)\text{-out}}(n)$.
	\end{observation}

	\begin{claim}\label{claim:1kout}
		Following $\mathcal S_{out}$ with $r=2$, w.h.p.\ $A_i \setminus \{i\}$ contains at least $k$ different vertices for every $1 \leq i \leq n$.
	\end{claim}
	
	\begin{proof}
		A block $A_i$ is called \emph{bad of type I} if some vertex $u \in [n]$ appears in $A_i$ at least three times. Similarly, a block $A_i$ is called \emph{bad of type II} if there are two distinct vertices $u, v \in [n]$, each appearing twice in $A_i$. Finally, a block is called \emph{bad} if it is bad of type I or II. For every $1 \leq i \leq n$, let $p_i$ denote the probability that $A_i$ is bad of type I. Clearly
		$$
		p_i \leq \binom{|A_i|}{3} \frac{1}{n^2} \leq \frac{k(k+1)(k+2)}{6 n^2}
		$$
		for every $1 \leq i \leq n$. Therefore the probability that there exists a block which is bad of type I is at most 
		$$
		\sum_{i=1}^n p_i \leq \frac{k(k+1)(k+2)}{6 n} = o(1).
		$$
		An analogous argument shows that the probability that there exists a block which is bad of type II is at most
		$$
		\sum_{i=1}^n \binom{|A_i|}{2} \binom{|A_i|-2}{2} \frac{1}{n^2} \leq \frac {k^2 (k+2)^2}{4n} = o(1).
		$$
		It follows that w.h.p.\ there are no bad blocks. We conclude that w.h.p.\ $A_i \setminus \{i\}$ contains at least $k$ different vertices for every $1 \leq i \leq n$ as claimed.
	\end{proof}
	
	%	. Let $M$ be the number of blocks $A_i$ such that there exists a vertex $u$ that was offered more than two times in $A_i$.  Let $p_{M}$ be the probability that a specific block $A_i$ is such that such that there exists a vertex $u$ that was offered more than two times in $A_i$. Then $p_{M}\leq \binom{|A_i|}{3}\frac 1{n^2}\leq \frac {k(k+1)(k+2)}{6n^2}$. Then $M$ is stochastically dominated by a random variable $M'\sim \Bin(n,\frac {k(k+1)(k+2)}{6n^2})$. Therefore, By Markov's inequality
	%	$$\Pr[M>0]\leq \Pr[M'\geq 0]=o(1).$$ Thus w.h.p.\ there are no such blocks. 
	
	%Similarly, let $T$ be the number of blocks $A_i$ for which there are two different vertices $u,v$ and each of them appears twice in $A_i$. Let $p_{T}$ be the probability that a specific block $A_i$ is such that there exist two vertices $u,v$ that each of them appear twice in $A_i$. Then $p_{T}\leq \binom{|A_i|}{2}\binom{|A_i|-2}{2}\frac 1{n^2}\leq \frac {(k+1)^2(k+2)^2}{4n^2}$. Then $T$ is stochastically dominated by a random variable $T'\sim \Bin(n,\frac {(k+1)^2(k+2)^2}{2n^2})$. Therefore, By Markov's inequality
	%	$$\Pr[T>0]\leq \Pr[T'\geq 0]=o(1).$$ Thus w.h.p.\ there are no such blocks. 
	
	%	We conclude that w.h.p.\ all blocks are such that there is at most one repeated vertex. Therefore, w.h.p.\  for every $i$, the block $A_i$ contains $k$ distinguished vertices, that are different from $v_i$.

	\begin{claim}\label{claim:2kout}
		Following $\mathcal S_{out}$ (with any fixed $k$ and $r$), w.h.p.\ $n' = o(n)$.
	\end{claim}
	
	\begin{proof}
		Our aim is to show that w.h.p.\ the number of big blocks is negligible compared to $n$. A block $A_i$ is called \emph{big of type I} if some vertex $u \in [n]$ appears in $A_i$ at least twice. Similarly, a block $A_i$ is called \emph{big of type II} if $i \in A_i$. Observe that a block is big if and only if it is big of type I or II. For every $1 \leq i \leq n$, let $p_i$ denote the probability that $A_i$ is big of type I. Clearly
		$$
		p_i \leq \binom{|A_i|}{2} \frac{1}{n} \leq \frac{(k+r)^2}{2n}
		$$
		for every $1 \leq i \leq n$. Let $X_1$ denote the number of blocks which are big of type I. Then $X_1$ is stochastically dominated by a random variable $Y_1 \sim \Bin \left(n, \frac{(k+r)^2}{2n} \right)$. Therefore, by Lemma~\ref{Che} we have
		$$
		\Pr[X_1 \geq \log n] \leq \Pr[Y_1 \geq \log n] \leq e^{- \log n} = o(1).
		$$ 
		Thus w.h.p.\ the number of blocks which are big of type I is at most $\log n$.  
		
		An analogous argument shows that w.h.p.\ the number of blocks which are big of type II is at most $\log n$ as well. We conclude that w.h.p.\ there are $o(n)$ big blocks. 
	\end{proof}
	
	%		Let $D$ be the number of blocks $A_i$ such that there exists a vertex $u$ that was offered two times in $A_i$.  Let $p_D$ is the probability that a specific block $A_i$ is such that such that there exists a vertex $u$ that was offered two times between the first $k$ elements of $A_i$. Then $p_D \leq \binom{k}{2} \frac{1}{n} \leq \frac{(k-1)k}{2n}$. Then $D$ is stochastically dominated by a random variable $D'\sim \Bin(n,\frac {(k-1)k}{2n})$. Therefore by Lemma~\ref{Che},
	%	$$
	%\Pr[D > \log \log n] \leq \Pr[D' \geq \log \log n] < e^{- \log \log n}.
	%		$$ 
	%	Thus w.h.p.\ the number of such blocks is at most $\log \log n$.

	%	Now, let $L$ be the number of blocks $A_i$ such that  $v_i$ was offered between the first $k$ elements of $A_i$.  Let $p_L$ is the probability that a specific block $A_i$ is such that  $v_i$ was offered between the first $k$ elements of $A_i$. Then $p_L\leq \binom{k}{1}\frac 1n\leq \frac {k}{n}$. Then $L$ is stochastically dominated by a random variable $L'\sim \Bin(n,\frac {k}{n})$. Therefore by Lemma~\ref{Che},
	%		$$\Pr[L>\log\log n]\leq \Pr[L'\geq \log\log n]<e^{-\log\log n}.$$ Thus w.h.p.\ the number of such blocks is at most $\log\log n$.
	
	%	We conclude that w.h.p.\ all but at most $n'=2\log\log n=o(n)$ blocks are such  that $v_i\notin A_i$, and every vertex in $A_i$ was offered at most once.  
	%\end{proof}
	
	\begin{claim}\label{claim:3kout}
		Let $k$ be a positive integer. Then $H \sim \mathcal G_{k\text{-out}}(n)$ and $H'\sim \mathcal G'_{(k,2)\text{-out}}(n)$ can be coupled	in such a way that w.h.p.\ $H\subseteq H'$.
	\end{claim}
	
	\begin{proof}
		Consider the oriented versions of $\mathcal G_{k\text{-out}}(n)$ and $\mathcal G'_{(k,2)\text{-out}}(n)$, denoted by $\vec{\mathcal G}_{k\text{-out}}(n)$ and by $\vec{\mathcal G}'_{(k,2)\text{-out}}(n)$, respectively. In order to prove the claim it suffices to show that $D \sim \vec{\mathcal G}_{k\text{-out}}(n)$ and $D'\sim \vec{\mathcal G}'_{(k,2)\text{-out}}(n)$ can be coupled in such a way that w.h.p.\ $D \subseteq D'$. 
		
		For every $1 \leq i \leq n$, let $X^{(i)}_1, X^{(i)}_{2}, \ldots$ be an infinite sequence of independent random variables, each having the uniform distribution on $[n]$. We can generate $D \sim \vec{\mathcal G}_{k\text{-out}}(n)$ as follows. For every vertex $1 \leq i \leq n$, let $\ell_i$ denote the smallest integer for which $\left \{X^{(i)}_1, X^{(i)}_{2}, \ldots, X^{(i)}_{\ell_i} \right\}$ contains exactly $k$ different elements from $[n] \setminus \{i\}$. These $k$ elements will be the $k$ out-neighbors of vertex $i$. To generate $D' \sim \vec{\mathcal G}'_{(k,2)\text{-out}}(n)$, for every vertex $1 \leq i \leq n$, we look at $\left\{X^{(i)}_1, X^{(i)}_{2}, \ldots, X^{(i)}_{k} \right\}$; if it contains exactly $k$ different elements from $[n] \setminus \{i\}$, then these $k$ elements will be the $k$ out-neighbors of vertex $i$. Otherwise, the $k+2$ elements of $\{X^{(i)}_1, X^{(i)}_{2}, \ldots, X^{(i)}_{k+2}\}$ will be the out-neighbors of vertex $i$ (as a directed multigraph).
		
		Since w.h.p.\ $\ell_i \leq k+2$ holds for every $1 \leq i \leq n$ by Observation~\ref{obs:kout} and by Claim~\ref{claim:1kout}, it follows that w.h.p.\ $D \subseteq D'$.
	\end{proof}
	
	Combining  Observation~\ref{obs:kout} with Claim~\ref{claim:3kout} implies that $H \sim \mathcal G_{k\text{-out}}(n)$ and $G \sim \mathcal S_{out}(n, k n + 2 n')$ can be coupled in such a way that w.h.p.\ $H \subseteq G$. Since by Claim~\ref{claim:2kout} we have that  $n' = o(n)$, this concludes the proof of Proposition~\ref{pro:Kout}.   
\end{proof}  

It was proved by Bohman and Frieze~\cite{bohman2009hamilton} that w.h.p.\ $G \sim \mathcal G_{\text{3-out}}(n)$ admits a Hamilton cycle (and, moreover, $3$ is the smallest integer for which this holds). Hence, it readily follows from Corollary~\ref{cor:Kout} that w.h.p.\ $\tau(\mathcal{H}, n) \leq (3 + o(1)) n$, where $\mathcal{H}$ is the property of admitting a Hamilton cycle (we would like to thank Michael Krivelevich for pointing this out). Similarly, for every fixed positive integer $k$, we have that w.h.p.\ $\tau(\cd_k,n) \leq (k + o(1)) n$, since the minimum degree of $\mathcal G_{k\text{-out}}(n)$ is at least $k$ (however, this bound is weaker than the result stated in Theorem~\ref{th::OnMinDegk}). In~\cite{frieze1986maximum}, Frieze extended the result of Walkup that we have mentioned in Subsection~\ref{subsec:connections}, by showing that w.h.p.\ $\mathcal G_{2\text{-out}}(n)$ admits a perfect matching, provided that $n$ is even. Combining this result with Corollary~\ref{cor:Kout}, implies that w.h.p.\ $\tau(\mathcal {PM},n) \leq (2 + o(1)) n$, where $\mathcal {PM}$ is the property of admitting a perfect matching. Frieze's result was further improved by Karo{\'n}ski, Overman and Pittel~\cite{KOP}, who proved that w.h.p.\ $\mathcal G_{(1+2e^{-1})\text{-out}}(K_{n,n})$ admits a perfect matching, where $\mathcal G_{(1+2e^{-1})\text{-out}}(K_{n,n})$ is obtained as follows. First, pick a random element of $\mathcal G_{1\text{-out}}(K_{n,n})$ and then give every vertex that has been chosen as a neighbor by at most one other vertex a `second chance' to pick another random neighbor. In the model $\mathcal G_{(1+e^{-1})\text{-out}}(K_{n,n})$, where only vertices that were not chosen at all are getting a
`second chance', w.h.p.\ there is no perfect matching~\cite{KOP}. We will discuss the games
$(\mathcal{H}, n)$ and $(\mathcal{PM}, n)$ further in Section~\ref{sec::openprob}.

In the remainder of this section we briefly explain how to approximate $\mathcal G_{(1+2e^{-1})\text{-out}}(K_{n,n})$ using our semi-random graph process. For simplicity we will restrict our attention to graphs with an even number of vertices.  

\begin{proposition}\label{pro:KoutBip}
	There exists a strategy $\mathcal S'_{out}$ for Builder such that $H \sim \mathcal G_{(1+2e^{-1})\text{-out}}(K_{n/2, n/2})$ and $G \sim \mathcal S'_{out}(n,(1+2e^{-1}+o(1))n)$ can be coupled in such a way that w.h.p.\ $H \subseteq G$. 
\end{proposition}

\begin{proof} %(sketch).
	Let $[n] = X_0 \cup X_1$ be an arbitrary equipartition; denote $X_0 = \left\{v^{(0)}_1, \ldots, v^{(0)}_{n/2} \right\}$ and $X_1 = \left\{v^{(1)}_1, \ldots, v^{(1)}_{n/2} \right\}$. For $i \in \{0,1\}$ and every positive integer $k$, let $x_i(k)$ denote the number of times a vertex of $X_i$ was offered during the first $k$ rounds; clearly $x_0(k) + x_1(k) = k$. We are now ready to describe Builder's strategy.
	%			
	%			Order the vertices by $v_1,\dots,v_n$ and divide the vertices into 2 sets of size $\lfloor\frac n2\rfloor$ (if $n$ is odd ignore the last vertex):
	%			$X_1=\{v_i \mid i\ is\ odd\},\ X_0=\{v_i \mid i\ is\ even\}$. Let $x_i(k)$ be the number of vertices that was offered from $X_i$ until the $k$th turn (and including it). Note that $x_0(k)+x_1(k)=k$.
	
	\medskip
	
	\noindent \textbf{Strategy $\mathcal S'_{out}$:} For every positive integer $k$, let $v^{(i)}_{j_k}$ (where $i \in \{0,1\}$ and $1 \leq j_k \leq n/2$) denote the vertex Builder is offered in round $k$. The strategy is divided into the following two phases.
	
	\smallskip
	
	\noindent \textbf{Phase I:} Let $f_0(n)$ be the number of rounds until $\min \{x_0(k), x_1(k)\} \geq n/2$ first occurs. For every $1 \leq k \leq f_0(n)$ Builder plays the $k$th round as follows:
	\begin{enumerate}[$(1)$]
		\item If $x_i(k) \leq n/2$, then he connects $v^{(i)}_{j_k}$ to $v^{(1-i)}_{x_i(k)}$ (that is, he connects the vertex he is offered to the $x_i(k)$th vertex from the other part). 
		\item If $x_i(k) > n/2$ but $x_{1-i}(k) < n/2$, then he connects $v^{(i)}_{j_k}$ to $v^{(1-i)}_{1}$.
	\end{enumerate}
	
	\noindent \textbf{Phase II:} At the beginning of this phase, for $i \in \{0,1\}$, let $Y_i$ be the set of vertices of $X_i$ that were offered at most once during Phase I. For $i \in \{0,1\}$ let $y_i = |Y_i|$ and let $Y_i = \left\{u^{(i)}_1, \ldots, u^{(i)}_{y_i} \right\}$. Let $y = \max \{y_0, y_1\}$. For $i \in \{0,1\}$ and every positive integer $k$, let $y_i(k)$ denote the number of times a vertex of $X_i$ was offered during the first $k$ rounds of Phase II. Let $f_1(n)$ be the number of rounds in Phase II until $\min \{y_0(k), y_1(k)\} \geq y$ first occurs. For every $1 \leq k \leq f_1(n)$ Builder plays the $k$th round of Phase II as follows:   
	\begin{enumerate}[$(i)$]
		\item If $y_i(k) \leq y$, then he connects $v^{(i)}_{j_k}$ to $u^{(1-i)}_{y_i(k)}$ (that is, he connects the vertex he is offered to the $y_i(k)$th vertex of $Y_{1-i}$). 
		\item  If $y_i(k) > y$ but $y_{1-i}(k) < y$, then he connects $v^{(i)}_{j_k}$ to $u^{(1-i)}_{1}$.
	\end{enumerate}

	Let $G \sim \mathcal S'_{out}(n, f_0(n) + f_1(n))$ and let $H$ be the graph obtained from $G$ by removing all the edges Builder has claimed in steps $(2)$ and $(ii)$. Then $H \sim \mathcal G_{(1+2e^{-1})\text{-out}}(K_{n/2,n/2})$.
	
	It remains to prove that w.h.p.\ $f_0(n) + f_1(n) = (1 + 2e^{-1} + o(1)) n$. Indeed, for $i \in \{0,1\}$, let $R_i$ denote the number of times a vertex of $X_i$ was offered during the first $n + \sqrt{n \log n}$ rounds of the game. Clearly $R_i \sim \Bin(n + \sqrt{n \log n}, 1/2)$ and thus $\Pr(R_i < n/2) \leq \Pr \left(R_i \leq \mathbb{E}(R_i) - \frac{\sqrt{n \log n}}{2} \right) = o(1)$ holds by Lemma~\ref{Che}. It follows that w.h.p.\ $f_0(n) \leq n + \sqrt{n \log n}$. 
	
	Similarly, for $i \in \{0,1\}$, let $N_i$ denote the number of vertices of $X_i$ that were offered at most once during the first $f_0(n)$ rounds. For every $1 \leq j \leq n/2$, let $I_j$ be the indicator random variable for the event ``$v_j^{(0)}$ was offered at most once during the first $f_0(n)$ rounds''. Then, for every $1 \leq j \leq n/2$, it holds that 
	$$
	\Pr(I_j = 1) = (1 - 1/n)^{f_0(n)} + f_0(n) \cdot 1/n \cdot (1 - 1/n)^{f_0(n)-1} = (2+o(1)) e^{-1}, 
	$$
	where the last equality holds by the concentration result we proved for $f_0(n)$ (note that, by definition, $f_0(n) \geq n$). Therefore
	$$
	\mathbb{E}(N_0) = \sum_{j=1}^{n/2} \mathbb{E}(I_j) = (1+o(1)) e^{-1} n.
	$$
	Observe that changing the offered vertex in any single round of the process can change $N_0$ by at most 1. Hence, applying Theorem~\ref{th::Azuma}, we deduce that 
	$$
	\Pr(|N_0 - \mathbb{E}(N_0)| \geq \sqrt{n \log n}) \leq 2 \exp \left\{- \frac{n \log n}{2 f_0(n)} \right\} = o(1),
	$$ 
	where the last equality holds by the concentration result we proved for $f_0(n)$. An analogous argument shows that $\mathbb{E}(N_1) = (1+o(1)) e^{-1} n$ and that $\Pr(|N_1 - \mathbb{E}(N_1)| \geq \sqrt{n \log n}) = o(1)$. Hence, w.h.p.\ $y = (1 + o(1)) y_i = (1 + o(1)) \mathbb{E}(N_i) = (1+o(1)) e^{-1} n$ for $i \in \{0,1\}$. Finally, an analogous argument to the one we used to prove the concentration of $R_i$, implies that w.h.p.\ $f_1(n) = (2 + o(1)) y = (2+o(1)) e^{-1} n$. We conclude that w.h.p.\ $f_0(n) + f_1(n) = (1 + 2e^{-1} + o(1)) n$ as claimed. 
\end{proof}

Proposition~\ref{pro:KoutBip} implies the following improved upper bound on the duration of the game $(\mathcal{PM}, n)$.

\begin{corollary} \label{cor:KoutBip}
	Assume that w.h.p.\ $H \sim \mathcal G_{(1+2e^{-1})\text{-out}}(n)$ satisfies the monotone increasing graph property $\mathcal P$. Then w.h.p.\ $\tau(\mathcal P, n) \leq (1+2e^{-1}+o(1)) n$. In particular, w.h.p.\ $\tau(\mathcal{PM}, n) \leq (1+2e^{-1}+o(1)) n$.
\end{corollary}

\subsection{The min-degree process}
\label{subsec:min_deg_process}

Recall the min-degree graph process. Let $G_{\min}(n, 0)$ be the empty graph with vertex set $[n]$ and, for every $m \geq 0$, let $G_{\min}(n, m+1)$ be obtained from $G_{\min}(n,m)$ by first choosing a vertex $u$ of minimum degree in $G_{\min}(n,m)$ u.a.r., and then connecting it by a new edge to a vertex $v \in [n] \setminus \{u\}$ chosen u.a.r.~among all vertices which are not connected to $u$ by an edge of $G_{\min}(n,m)$.

For proving Proposition~\ref{pro:Gmin}, we will use the following related models $\{G'_{\min}(n,m)\}_{m \geq 0}$ and $\{G''_{\min}(n,m)\}_{m \geq 0}$ which are defined as follows. Let $\{G'_{\min}(n,m)\}_{m \geq 0}$ be the same as $\{G_{\min}(n,m)\}_{m \geq 0}$ except that we allow multiple edges, that is, we choose $v$ u.a.r.~among all vertices of $[n] \setminus \{u\}$. Similarly, $\{G''_{\min}(n,m)\}_{m \geq 0}$ is the same as $\{G_{\min}(n,m)\}_{m \geq 0}$ except that we allow loops and multiple edges, that is, we choose $v$ u.a.r.~among all vertices of $[n]$.  We first prove that our semi-random multigraph process can be used to generate $\{G''_{\min}(n,m)\}_{m \geq 0}$. 

\begin{proposition}\label{pro:G''min}
	There exists a strategy $\mathcal S_{min}$ for Builder such that the probability space $\mathcal S_{min}(n,m)$ is the same as the probability space $G''_{\min}(n,m)$.
\end{proposition}

The semi-random process can also be used to approximate $\{G'_{min}(n,m)\}_{m \geq 0}$.

\begin{proposition}\label{pro:G'min}
	If $m = o \left(n^2 \right)$, then the strategy $\mathcal S_{min}$ is such that $H \sim G'_{\min}(n,m)$ and $G \sim \mathcal S_{min}(n, (1+o(1))m)$ can be coupled
	in such a way that w.h.p.\ $H \subseteq G$.
\end{proposition}

The following are immediate corollaries of Propositions~\ref{pro:G''min} and~\ref{pro:G'min}.

\begin{corollary}\label{cor:G''min}
	Let $m_{\mathcal P}$ be a positive integer for which w.h.p.\ $H \sim G''_{min}(n,m_{\mathcal P})$ satisfies the monotone increasing graph property $\mathcal P$. Then w.h.p.\ $\tau(\mathcal P, n) \leq m_{\mathcal P}$.
\end{corollary}

\begin{corollary}\label{cor:G'min}
	Let $m_{\mathcal P}$ be a positive integer for which w.h.p.\ $H \sim G'_{min}(n,m_{\mathcal P})$ satisfies the monotone increasing graph property $\mathcal P$. If $m_{\mathcal P} = o \left(n^2 \right)$, then w.h.p.\ $\tau(\mathcal P, n) \leq (1+o(1)) m_{\mathcal P}$.
\end{corollary}

\begin{proof}[Proof of Proposition~\ref{pro:G''min}] 
The strategy $\mathcal S_{\min}$ used by Builder is the following.
 
\textbf{Strategy $\mathcal S_{min}$:}	Whenever Builder is offered some vertex $v$, he connects it to a vertex $u$, chosen u.a.r.~among all vertices of minimum degree (observe that this could result in loops and multiple edges). For this purpose, the degree of $v$ in Builder's graph is computed \textit{before} it is offered.

	Fix a non-negative integer $r$ (that may depend on $n$), an arbitrary multigraph $G$ with vertex set $[n]$ and $r$ edges, and arbitrary indices $1 \leq i, j \leq n$. In order to prove that our process generates $\{G''_{\min}(n,m)\}_{m \geq 0}$, it suffices to prove that the probability of $ij$ being added to $G$ in round $r+1$ of our process is equal to the probability of $ij$ being added to $G$ in round $r+1$ of $\{G''_{\min}(n,m)\}_{m \geq 0}$. 
	
	Indeed, in $\{G''_{\min}(n, m)\}$ the first vertex of the added edge is picked uniformly at random from those vertices that have minimum degree before the round, and the second vertex is picked u.a.r. from all vertices. On the other hand, in our process the offered vertex is picked u.a.r. from all vertices, while the second vertex is picked u.a.r. from the vertices that had minimum degree before the round. The statement follows.
\end{proof}

We are now ready to prove Proposition~\ref{pro:Gmin}.
Note that a similar and, in fact slightly simpler process, can be used to approximate $\{G'_{\min}(n,m)\}_{m \geq 0}$ (thus proving Proposition~\ref{pro:G'min}). 

\begin{proof}[Proof of Proposition~\ref{pro:Gmin}]
	We apply $\mathcal S_{\min}$ as described in Subsection \ref{subsec:connections}. Whenever Builder claims a loop or a multiple edge, he ignores this edge and considers this round to be a failure. We continue running this process until there are $m$ edges in Builder's graph which are not failures.
	That is, for every $m$ we will be able to generate $G_{\min}(n,m)$ by $\mathcal S_{\min}(n,m+f(m))$, where $f(m)$ is the number of failures. We now prove that, if $m = o(n^2)$, then w.h.p. $f(m) = o(m)$. Consider a specific round of our process which starts with a (simple) graph $G$ with $r$ edges. Let $j$ denote the second vertex we choose in this round. By the description of our process, we must have $d_G(j) = \delta(G) \leq 2r/n$. Hence, the probability that the first vertex we choose in this round is $j$ or one of its neighbors in $G$ is at most $\frac{1 + 2r/n}{n} = \frac{1}{n} + \frac{2r}{n^2}$. For every positive integer $m$, let $Y_m$ denote the random variable which counts the number of failures that occur during the first $m$ rounds of running our process according to $\mathcal S_{\min}$. Then
	$$
	\mathbb{E}(Y_m) \leq \frac{m}{n} + \sum_{r=1}^m \frac{2r}{n^2} = O(m/n + m^2/n^2) = o(m),
	$$
	where the last equality holds by our assumption that $m = o(n^2)$. Applying Markov's inequality to $Y_m$, we conclude that indeed w.h.p. the number of failures is $o(m)$.
\end{proof}

\section{Offline games} \label{sec:off}
	
	Theorems~\ref{thm:OffFix} and~\ref{th::MinDegkOffline} (as well as a few other results which will be discussed in Section~\ref{sec::openprob}) are consequences of a general result. Before stating it, we need to introduce some notation.
	For a directed graph $D$ with vertex set $\{v_1, \ldots, v_r\}$, let $d_i^+$ denote the out-degree of $v_i$ in $D$ for every $1 \leq i \leq r$. For a given sequence $S=\{v_i\}_{i\in \mathbb{N}}$, let $m(D)$ denote the smallest integer $j$ such that  in the subsequence $S'=(v_1,v_2,\dots,v_j)$ there are $r$ distinct vertices $u_1, \ldots, u_r \in [n]$ so that for every $1 \leq i \leq r$, $u_i$ appears at least $d_i^+$ times in $S'$. For an undirected graph $H$, let $m(H) = \min \{m(D) : D \textrm{ is an orientation of } H\}$.  
	%We prove proposition for general graphs and deduce a corollary for graphs $H$ of fixed size.
	
	\begin{proposition}\label{prop:FindingH}
		Let $H$ be a graph on at most $n$ vertices, let $S=\{v_i\}_{i\in \mathbb{N}}$ be a sequence of vertices from $[n]$, chosen independently and uniformly at random with replacement, and let $\mathcal P_H$ be the graph property of containing $H$ as a subgraph. Then $\tau'(\mathcal{P}_H, n) = m(H)$.	
	\end{proposition}

\begin{proof}[Proof of Proposition \ref{prop:FindingH}]
	Starting with the lower bound, suppose that Builder has a strategy $\mathcal{S}$ to construct a copy of $H$ in $\ell$ rounds. During the game, played according to $\mathcal{S}$, orient each edge claimed by Builder from the vertex he was offered to the vertex he chose. Fix some copy of $H$ in $G_{\ell}$ and let $D'$ be its orientation according to the aforementioned rule. Then $\ell \geq m(D') \geq m(H)$.
	
	As for the upper bound, we will describe a strategy for Builder to construct a copy of $H$ in $m(H)$ rounds. Fix an arbitrary orientation $D$ of $H$ such that $m(D) = m(H)$. Let $\{v_1, \ldots, v_r\}$ denote the vertex set of $H$ and let $u_1, \ldots, u_r$ be vertices in $[n]$ such that $\off(u_i) \geq d_D^+(v_i)$ for every $1 \leq i \leq r$. For every $1 \leq i \leq r$ and $1 \leq j \leq d_D^+(v_i)$, let $v_{i,1}, \ldots, v_{i,d^+(v_i)}$ denote the out-neighbors of $v_i$ in $D$. Let $\varphi$ be the function which maps $v_i$ to $u_i$ for every $1 \leq i \leq r$. In every round, if there exist $1 \leq i \leq r$ and $1 \leq j \leq d_D^+(v_i)$ such that the vertex offered in this round is $u_i$ and this is the $j$th time it is offered, then Builder claims the edge $(u_i, \varphi(v_{i,j}))$, if it is free. In any other case, he plays arbitrarily. It is evident that, if Builder follows this strategy, then, after $m(H)$ rounds, $G_m[\{u_1, \ldots, u_r\}]$ contains a copy of $H$.
\end{proof}

%%%%%%%%%%%%%%%%%%%%%%%%%%%%%%%%%%%%%%%%%%%%%%%%%%%%%%%%%%%%%%%%%%%%%%%%%%%%%%%

\subsection{Offline fixed graphs} \label{subsec::OffFixGraph}
%Let $H$ be a fixed graph. The following is a corollary of Proposition~\ref{prop:FindingH}.

\begin{proof}[Proof of Theorem~\ref{thm:OffFix}]
	Let $r, f$, and $g$ be as in the statement of the theorem. It follows by Proposition~\ref{prop:FindingH} that $m(H)=\tau'(\mathcal{P}_H, n)$ and so it remains to prove that w.h.p.\ $f(n) \cdot n^{(r-1)/r} \leq m(H) \leq g(n) \cdot n^{(r-1)/r}$. Fix some $m < f(n) \cdot n^{(r-1)/r}$. The expected number of vertices which were offered at least $r$ times during the first $m$ rounds is
	\begin{eqnarray} \label{eq::fnMarkov}
	 \sum_{k = r}^m f(n,m,k) &=& \sum_{k = r}^m n \binom{m}{k} \left(\frac{1}{n} \right)^k \left(1 - \frac{1}{n} \right)^{m-k}
	\leq \sum_{k = r}^m \left(\frac{e}{k} \right)^k \frac{m^k}{n^{k-1}} \nonumber \\
	&\leq& 3 \sum_{k = r}^m (f(n))^k \cdot n^{(k(r-1)/r) - k + 1} \leq 3 (f(n))^r + \sum_{k = r+1}^m n^{1 - k/r} = o(1).
	\end{eqnarray}
	
	Since every orientation of $H$ contains a vertex of out-degree at least $r$, it follows from~\eqref{eq::fnMarkov} and from Markov's inequality that w.h.p. $m(H) > m$.
	
	Next, set $m = g(n) \cdot n^{(r-1)/r}$. Clearly we can assume that $g(n) \leq \log n$ and thus, in particular, $m = o(n^{r/(r+1)})$. Hence, a simple calculation which is very similar to the one in~\eqref{eq::fnMarkov}, shows that w.h.p. no vertex was offered more than $r$ times during the first $m$ rounds. On the other hand, the expected number of vertices which were offered exactly $r$ times is
	$$
	f(n,m,r) = (1 + o(1)) \frac{e^{- m/n}}{r!} \cdot \frac{m^r}{n^{r-1}} \geq C (g(n))^r = \omega(1),
	$$
	where the first equality holds by~\eqref{e:exactly-k} from Subsection~\ref{sec::BallsBins} and $C$ is some constant which depends on $r$. Therefore, by Lemma~\ref{lem::BallsTalagrand}, w.h.p. at least $v(H)$ vertices were offered exactly $r$ times each. Hence, $m(H) \leq m$ as claimed. \end{proof}

\subsection{Offline Minimum Degree $k$}

\begin{proof}[Proof of Theorem~\ref{th::MinDegkOffline}]
For a positive integer $k$ and a vertex $v \in V(K_n)$, at any point during the game $({\cd}_k, n)$,
%let $\off(v)$ denote the number of times $v$ was offered up to that point and
let $\off_k(v) = \min \{2 \off(v), k + \off(v)\}$, where $\off(v)$ is the number of times $v$ was offered up to that point. The idea behind this definition of $\off_k(v)$ is that when Builder claims an edge which is incident with $v$, he advances the sum of degrees in his graph towards having minimum degree $k$ by $2$ if $d(v) < k$, and by $1$ otherwise (in both cases, this is only if he chooses the other endpoint of this edge wisely). Therefore, this parameter allows us to track the evolution of Builder's graph and show that, in the offline version of the game $({\cd}_k, n)$, Builder can construct a graph with minimum degree $k$ as soon as $\sum_{v \in V(K_n)} \off_k(v) \geq k n$ first occurs, but not sooner. We begin by stating and proving the following two simple auxiliary claims.

\begin{claim} \label{cl::Yk}
Let $Y_k^r = \sum_{i=0}^{k-1} (k - i) X_i^r$, where $X_i^r$ is the random variable which counts the number of vertices that were offered precisely $i$ times during the first $r$ rounds. Then $Y_k^r \leq r$ if and only if $\sum_{v \in V(K_n)} \off_k(v) \geq k n$.
\end{claim}

\begin{proof}
Our claim readily follows from the following calculation:
\begin{align*}
\sum_{v \in V(K_n)} \off_k(v) &= \sum_{\stackrel{v \in V(K_n)}{\off(v) \leq k}} 2 \off(v) + \sum_{\stackrel{v \in V(K_n)}{\off(v) > k}} (k + \off(v)) \\
&= 2 \sum_{i=0}^k i X_i^r + k \sum_{i = k+1}^r X_i^r + \sum_{i = k+1}^r i X_i^r \\
&= \sum_{i=0}^r i X_i^r + k \sum_{i=0}^r X_i^r - \sum_{i=0}^{k-1} (k - i) X_i^r \\
&= r + k n - Y_k^r.
\end{align*}
\end{proof}

\begin{claim} \label{cl::dk}
For a vertex $v \in V(K_n)$, at any point during the process, let $d_k(v) = \min \{d(v), k\}$. Then $\sum_{v \in V(K_n)} \off_k(v) \geq \sum_{v \in V(K_n)} d_k(v)$ holds at any point during the process.
\end{claim}

\begin{proof}
We will prove the claim by induction on the number of rounds in the process. It is clearly true at the beginning of the process when $\sum_{v \in V(K_n)} \off_k(v) = \sum_{v \in V(K_n)} d_k(v) = 0$. Assume that $\sum_{v \in V(K_n)} \off_k(v) \geq \sum_{v \in V(K_n)} d_k(v)$ holds immediately after the $i$th round of the process, for some $i \geq 0$, and consider the $(i+1)$st round. Let $u$ be the vertex Builder is offered in the $(i+1)$st round and let $v$ be the vertex he connects to $u$ in this round. If $u$ has degree at least $k$ in the beginning of the $(i+1)$st round, then $\off_k(u)$ is increased by $1$ and $\off_k(x)$ is unchanged for every $x \in V(K_n) \setminus \{u\}$. Moreover, regardless of Builder's strategy, $d_k(v)$ is increased by at most $1$ and $d_k(x)$ is unchanged for every $x \in V(K_n) \setminus \{v\}$. Similarly, if $u$ has degree at most $k-1$ in the beginning of the $(i+1)$st round, then $\off_k(u)$ is increased by $2$ and $\off_k(x)$ is unchanged for every $x \in V(K_n) \setminus \{u\}$. Moreover, regardless of Builder's strategy, $d_k(u)$ is increased by $1$, $d_k(v)$ is increased by at most $1$, and $d_k(x)$ is unchanged for every $x \in V(K_n) \setminus \{u,v\}$. In both cases $\sum_{v \in V(K_n)} \off_k(v) \geq \sum_{v \in V(K_n)} d_k(v)$ holds immediately after the $(i+1)$st round.
\end{proof}

Returning to the proof of Theorem~\ref{th::MinDegkOffline}, assume first that $\sum_{v \in V(K_n)} \off_k(v) < k n$. It follows from Claim~\ref{cl::dk} that $\sum_{v \in V(K_n)} d_k(v) \leq \sum_{v \in V(K_n)} \off_k(v) < k n$. Clearly this means that the minimum degree in Builder's graph is strictly less than $k$.

Assume now that $\sum_{v \in V(K_n)} \off_k(v) \geq k n$ holds after $r$ rounds of the process. We will show that Builder has a strategy to ensure that the minimum degree in his graph will be at least $k$. For every $1 \leq i \leq r$, let $u_i$ denote the vertex which Builder was offered in the $i$th round. Immediately after the $r$th round, for every $0 \leq j \leq k-1$ let $L_j$ denote the set of vertices which were offered exactly $j$ times. We will construct $k$ bipartite graphs $H_1 = (A_1 \cup B_1, E_1), \ldots, H_k = (A_k \cup B_k, E_k)$. Moreover, for every $1 \leq j \leq k$, we will construct a matching $M_j$ in $H_j$ which saturates $B_j$. This is done recursively as follows. $A_1 = \{u_1, \ldots, u_r\}$ (note that $A_1$ is a multiset of size $r$ in which every $v \in V(K_n)$ appears precisely $\off(v)$ times), $B_1 = L_0 \cup \ldots \cup L_{k-1}$ and $E_1 = \{uv : u \in A_1, v \in B_1 \textrm{ and } u \neq v\}$. Note that $|B_1| = \sum_{j=0}^{k-1} |L_j| = \sum_{j=0}^{k-1} X_j^r \leq Y_k^r \leq r = |A_1|$, where the last inequality holds by Claim~\ref{cl::Yk}. Let $u \in A_1$ and $v \in B_1$ be arbitrary vertices and let $0 \leq j \leq k-1$ be the unique integer such that $v \in L_j$. By definition, if $uv \notin E_1$, then $u = v$. Since $v \in L_j$, it follows that $d_{H_1}(v) = |A_1| - j \geq |A_1| - k$. Similarly, $d_{H_1}(u) \geq |B_1| - 1$. Since, moreover, $r \gg k$ for sufficiently large $n$, a straightforward application of Hall's Theorem shows that $H_1$ has a matching which saturates $B_1$; let $M_1$ be such a matching chosen arbitrarily. Assume we have already constructed $H_1, \ldots, H_i$ and $M_1, \ldots, M_i$ for some $1 \leq i < k$. Let $Z_i$ denote the set of vertices of $A_i$ that were matched in $M_i$ and let $A_{i+1} = A_i \setminus Z_i$ (again, $A_{i+1}$ is a multiset and so, if some vertex appears $\ell_1$ times in $A_i$ and $\ell_2$ in $Z_i$, then it will appear $\ell_1 - \ell_2$ times in $A_{i+1}$). Let $B_{i+1} = L_0 \cup \ldots \cup L_{k-i-1}$. Let $F_i^1 = \{uv \in E_i : u \in Z_i \textrm{ or } v \in L_{k-i}\}$ and let $F_i^2 = \{uv \in E_i : \exists u' \in Z_i \textrm{ such that } u'v \in M_i \textrm{ and } u' \textrm{ is a copy of } u\}$. Finally, let $E_{i+1} = E_i \setminus (F_i^1 \cup F_i^2)$. It remains to prove that $H_{i+1}$ has a matching which saturates $B_{i+1}$. We first claim that $|B_{i+1}| \leq |A_{i+1}|$. Indeed
\begin{eqnarray*}
|A_{i+1}| - |B_{i+1}| &=& |A_1| - \sum_{j=1}^i |M_j| - |B_{i+1}| = r - \sum_{j=1}^{i+1} |B_j| \\
&=& r - \sum_{j=1}^i j |L_{k-j}| - (i+1) \sum_{j=i+1}^k |L_{k-j}| \geq r - \sum_{j=0}^{k-1} (k-j) X_j^r \\
&=& r - Y_k^r \geq 0,
\end{eqnarray*}
where the last inequality holds by Claim~\ref{cl::Yk}.

Now, consider arbitrary vertices $u \in A_{i+1}$ and $v \in B_{i+1}$. By definition, if $uv \notin E_{i+1}$, then either $u = v$, or $uv \in M_1 \cup \ldots \cup M_i$, or $u'v \in M_1 \cup \ldots \cup M_i$ for some $u'$ which is a copy of $u$. Therefore
$$
d_{H_{i+1}}(v) \geq |A_{i+1}| - \off(v) - i \cdot \max \{\off(z) : z \in V(K_n)\} \geq |A_{i+1}| - k \log n,
$$
where the last inequality holds by Corollary~\ref{cor::maxLoad}. Similarly,
$$
d_{H_{i+1}}(u) \geq |B_{i+1}| - 1 - \off(u) \geq |B_{i+1}| - 1 - \log n.
$$
Since, moreover, $|A_{i+1}| \geq |B_{i+1}| \gg \log n$ for sufficiently large $n$, it follows by a straightforward application of Hall's Theorem that $H_{i+1}$ indeed has a matching which saturates $B_{i+1}$; let $M_{i+1}$ be such a matching chosen arbitrarily. This proves the existence of the matching $M_j$ in $H_j$ for every $1 \leq j \leq k$.

We are now ready to describe Builder's strategy. For every $1 \leq i \leq r$, Builder plays the $i$th round as follows. If $u_i$ is unmatched in $M_1 \cup \ldots \cup M_k$, then Builder claims an arbitrary free edge which is incident to $u_i$. Otherwise, let $1 \leq j \leq k$ be the unique integer such that $u_i$ is matched in $M_j$. Builder claims the unique edge of $M_j$ which is incident to $u_i$ if it is free, and an arbitrary free edge which is incident to $u_i$ otherwise. We claim that, by following this strategy, after $r$ rounds the minimum degree in Builder's graph is at least $k$. Indeed, let $v \in V(K_n)$ be an arbitrary vertex. If $\off(v) \geq k$, then the degree of $v$ in Builder's graph is at least $k$ regardless of his strategy. Assume then that $\off(v) = \ell$ for some $0 \leq \ell \leq k-1$. By Builder's strategy, $v$ is matched in $M_j$ for every $1 \leq j \leq k - \ell$. Hence, its degree in Builder's graph is at least $\ell + (k - \ell) = k$.

Finally, we are in a position to determine $\tau({\cd}_k, n)$ asymptotically. It is the smallest integer $r$ for which $\sum_{v \in V(K_n)} \off_k(v) \geq k n$. By Claim~\ref{cl::Yk} it is then also the smallest $r$ for which $Y_k^r \leq r$. It follows by Lemma~\ref{lem::Balls} that $Y_k^r$ is concentrated around its expectation and thus $(1 + o(1)) \E(Y_k^r) \leq r$. By linearity of expectaion and by~\eqref{e:exactly-k} the latter inequality translates to
$$
\alpha n \geq (1 + o(1)) \E(Y_k^r) = (1 + o(1)) \sum_{i=0}^{k-1} (k - i) \E(X_i^r) = (1 + o(1)) \sum_{i=0}^{k-1} \left(k - i \right) \frac{e^{- \alpha}}{i!} \alpha^i n,
$$
where $\alpha := r/n$. Since $f_k(x) = \sum_{i=0}^{k-1} \left(k-i \right) \frac{x^i}{i!} - x e^x$ is a continuous function, it follows that $\alpha = \alpha_k + o(1)$ where $\alpha_k$ is the unique positive root of $f_k(x)$, as claimed.
\end{proof}

%%%%%%%%%%%%%%%%%%%%%%%%%%%%%%%%%%%%%%%%%%%%%%%%%%%%%%%%%%%%%%%%%%%%%%%%%%%%%%%

\section{Online games}\label{sec:on}

\subsection{Online fixed graph game}

\begin{proof}[Proof of Theorem~\ref{th::fixedGraphUpperBound}]
	We will describe a strategy for Builder to build a copy of $H$ and will then prove that w.h.p.\ building such a copy using this strategy will take him at most $g(n) \cdot n^{(d-1)/d}$ rounds. Let $(v_1, \ldots, v_r)$ be a degeneracy ordering of the vertices of $H$, that is, an ordering such that $v_k$ has at most $d$ neighbors in $\{v_1, \ldots, v_{k-1}\}$ for every $2 \leq k \leq r$. We refer to these neighbors as the \textit{back neighbors} of $v_k$. We will define a mapping $\varphi : V(H) \to [n]$ such that $G_{m_k}[\{\varphi(v_1), \ldots, \varphi(v_k)\}]$ contains a copy of $H[\{v_1, \ldots, v_k\}]$ for every $1 \leq k \leq r$, where $m_k$ is the number of the round in which this is achieved for the first time. We will do so inductively as follows. Let $u_1$ be the vertex Builder is offered in the first round of the game and let $\varphi(v_1) = u_1$. Assume now that for some $1 \leq k \leq r-1$, Builder has already built a graph $G_{m_k}$ and defined $\varphi(v_i)$ for every $1 \leq i \leq k$ such that $G_{m_k}[\{\varphi(v_1), \ldots, \varphi(v_k)\}]$ contains a copy of $H[\{v_1, \ldots, v_k\}]$. Builder would now wish to define $\varphi(v_{k+1})$; he does so as follows. Let $v_{i_1}, \ldots, v_{i_\ell}$ be the back neighbors of $v_{k+1}$ in $H$. 
	If $\ell = 0$, then Builder defines $\varphi(v_{k+1}) = u$ for an
	arbitrary vertex $u \in [n] \setminus \{\varphi(v_1), \ldots,
	\varphi(v_k)\}$ and $m_{k+1} = m_k$. Assume then that $1 \leq \ell
	\leq d$ and observe that $\varphi(v_{i_j})$ was already defined for
	every $1 \leq j \leq d$.
	 For every $i > m_k$, let $u_i$ denote the vertex Builder is offered in the $i$th round. If $u_i \in \{\varphi(v_1), \ldots, \varphi(v_k)\}$, then Builder plays arbitrarily. Otherwise, let $t$ denote the total number of times $u_i$ was offered in rounds $m_k + 1, \ldots, i$. Builder claims the edge $u_i \varphi(v_{i_t})$. If, moreover, $t = \ell$, then Builder defines $\varphi(v_{k+1}) = u_i$ and $m_{k+1} = i$ (in particular, $1 \leq t \leq \ell$ and so this strategy is well-defined).  
	
	Now, observe that $m_{k+1} - m_k$ is the smallest number of rounds until some $u_i \in [n] \setminus \{\varphi(v_1), \ldots, \varphi(v_k)\}$ is offered $\ell$ times (counting offers in rounds $m_k + 1, \ldots, m_{k+1}$). By Builder's strategy, no vertex is offered more than $\ell$ times during those $m_{k+1} - m_k$ rounds and so the conditions of Lemma~\ref{lem::BallsTalagrand} are satisfied. Since, moreover, $r$ is a constant and $\ell \leq d$, it follows from Lemma~\ref{lem::BallsTalagrand} that w.h.p.\ $m_{k+1} - m_k <g(n)/r \cdot n^{(d-1)/d}$. This is true for every $1 \leq k < r$ (and $m_1 = 1$) and so w.h.p.\ the entire game lasts at most $g(n) \cdot n^{{(d-1)}/{d}}$ rounds as claimed.       
\end{proof}

\begin{proof}[Proof of Theorem~\ref{thm:lowerCliqueOnLine}]
	The assertion of the theorem trivially holds for $d=2$. Hence, for the remainder of the proof we may assume that $d \geq 3$.
	 The main ingredient of the proof is the following claim which upper bounds the number of copies of $K_\ell$, for some integer $\ell\geq3$, in Builder's graph up to some specific round.
%	Observe that if in the $i$th round of the game, a new copy $G'$ of $K_\ell$ was created, then in this round Builder was offered a vertex $v\in V(G')$ which lies in a copy $G''$ of $K_{\ell-1}$ in $G'$. 
%	Furthermore, assume that in some point during the game, a copy $G''$ of $K_{\ell-1}$ appears in the graph $G$. Then for every pair $(K_{\ell-2},v)_{G''}$, if $v$ was offered in the $i$th round, then there is at most one copy $G'$ of $K_{\ell}$ that was created in the $i$th round and contains all the vertices of the pair (that is, all vertices of $G''$).
	
	\begin{claim} \label{claim:NumberOfCopiesOfKk}
		For positive integers $\ell$ and $m = m(n)$ and a strategy $\mathcal S$ of Builder, let $Z_{m, \ell}^{\mathcal S}$ denote the number of copies of $K_\ell$ in $G_m$ when Builder is playing according to the strategy ${\mathcal S}$. Then w.h.p.\  $Z_{m,\ell}^{\mathcal S} \leq a(n) \cdot \frac {m^{\ell-1}}{n^{\ell-2}}$, for any function $a(n)$ which tends to infinity as $n$ tends to infinity.
	\end{claim}
	
	\begin{proof}
	For positive integers $\ell \geq 3$ and $m = m(n)$ and a strategy ${\mathcal S}$, let $Y_{m, \ell}^{\mathcal S}$ denote the number of copies of $K_{\ell}$ that Builder creates in the $m$th round when he plays according to ${\mathcal S}$. Note that $Z^{\mathcal S}_{m,\ell}=\sum_{i=1}^mY^{\mathcal S}_{i,\ell}$.
		
		Let ${\mathcal S}$ be an arbitrary strategy of Builder. We prove by induction on $\ell$ that $\E(Y^{\mathcal S}_{m,\ell}) \leq \left(\frac {(\ell-1) m}{n} \right)^{\ell-2}$. First we introduce some useful notation. 	For an integer $m$, let $G'$ be a copy of $K_\ell$ in $G_m$. For $v\in V(G') \subseteq V(G_m)$, we denote by $(K_{\ell-1},v)_{G'}$ the ordered pair consisting of the copy of $K_{\ell-1}$ in $G'$ that does not contain $v$, and the remaining vertex $v$. For the induction basis $\ell = 3$, it is evident that in order to create a copy of $K_3$ in some round $r$, we must touch a vertex $v$ that belongs to a copy $G'$ of $K_2$. If $v$ was offered to Builder in round $r$, then the pair $(K_1,v)_{G'}$ will lie in at most one copy of $K_3$ after round $r$ (it is impossible to create two different copies of $K_3$, both containing $G'$, in one round). The number of such potential pairs after $m$ rounds is at most $2m$ and the probability that a specific vertex $v$ will be offered in round $r$ is $\frac 1n$. Therefore,  $\E(Y^{\mathcal S}_{m,3})\leq 2m\cdot \frac 1n=\frac {2m}{n}$. Now, for the induction step, assume that $\E(Y^{\mathcal S}_{m,\ell}) \leq \left(\frac {(\ell-1) m}{n} \right)^{\ell-2}$. In order to create a copy of $K_{\ell+1}$, Builder  must touch a vertex $v$ that belongs to a copy $G'$ of $K_\ell$. If $v$ was offered to Builder in some round $r$, then the pair $(K_{\ell-1},v)_{G'}$ will lie in at most  one copy of $K_{\ell+1}$ immediately after the $r$th round. Since  for each copy $G'$ of $K_\ell$ there are $\ell$  ordered pairs $(K_{\ell-1},v)_{G'}$, the expected number of such potential pairs after $m-1$ rounds is ${\ell}\cdot\E(Z^{\mathcal S}_{m-1,\ell})$ and the probability that a specific vertex $v$ will be offered in round $r$ is $\frac 1n$. Therefore,
		\begin{align*}
		\E(Y^{\mathcal S}_{m,\ell+1})&\leq\ell\cdot\E(Z^{\mathcal S}_{m,\ell})\cdot \frac 1n
		=\frac \ell n \sum_{i=1}^m \E(Y^{\mathcal S}_{i,\ell})\\
		&\leq \frac \ell n \sum_{i=1}^m (\ell-1)^{\ell-2}\left(\frac {i}n\right)^{\ell-2}
		=\frac {\ell(\ell-1)^{\ell-2}}{n^{\ell-1}}\sum_{i=1}^{m}i^{\ell-2}\\
		&\leq \frac {\ell(\ell-1)^{\ell-2}}{n^{\ell-1}}\cdot m^{\ell-1}
		\leq \frac {\ell^{\ell-1}}{n^{\ell-1}}\cdot m^{\ell-1}.
		\end{align*}
		This concludes the induction, from which it follows that 
		\begin{equation}\label{eq:LowerOnLineFix}
		\E(Z^{\mathcal S}_{m,\ell})=\sum_{i=1}^m \E(Y^{\mathcal S}_{i,\ell})\leq (\ell-1)^{\ell-2}\cdot\frac {m^{\ell-1}}{n^{\ell-2}}.
		\end{equation}
		Now, let ${\mathcal S}^m_{op}$ be a strategy for Builder which maximizes the number of copies of $K_\ell$ in $G_m$. 
%		where his goal is to have at least $a(n)\cdot \frac {m^{\ell-1}}{n^{\ell-2}}$ copies of $K_\ell$ in $G_m$, where $a(n)$ is any function which tends to infinity with $n$. 
		Let $Z_{m,\ell}^{op}$ be the number of copies of $K_\ell$ in $G_m$ when Builder plays according to ${\mathcal S}^m_{op}$. It follows by \eqref{eq:LowerOnLineFix} that $\E(Z_{m,\ell}^{op})\leq C\cdot\frac {m^{\ell-1}}{n^{\ell-2}}$, where $C = C(\ell)$ is a constant. By Markov's inequality we have that
		$$\lim_{n\to\infty}\Pr\left[Z_{m,\ell}^{op}>a(n)\cdot \frac {m^{\ell-1}}{n^{\ell-2}} \right]= 0$$ for any function $a(n)\to \infty$. Since  $$\Pr\left[Z_{m,\ell}^{{\mathcal S}}>a(n)\cdot \frac {m^{\ell-1}}{n^{\ell-2}} \right]\leq \Pr\left[Z_{m,\ell}^{op}>a(n)\cdot \frac {m^{\ell-1}}{n^{\ell-2}} \right]$$ holds for every strategy ${\mathcal S}$, our claim follows.
	\end{proof}

	Now, let $m =
	(a(n))^{-1} n^{(d-2)/(d-1)}$. Then, by Claim~\ref{claim:NumberOfCopiesOfKk}, w.h.p.
	$Z_{m,d}^{\mathcal S} < 1$ for any strategy ${\mathcal S}$. That is,
	if $m \leq (a(n))^{-1} n^{(d-2)/(d-1)}$, then w.h.p. $G_m$ does not
	contain a copy of $K_d$.
\end{proof}

%%%%%%%%%%%%%%%%%%%%%%%%%%%%%%%%%%%%%%%%%%%%%%%%%%%%%%%%%%%%%%%%%%%%%%%%%%%%%%%

\subsection{Online Minimum Degree $k$} \label{subsec:online_min_deg_k}

Our main goal in this subsection is to prove Theorem~\ref{th::OnMinDegk}. In fact, we will study three variants of the minimum degree $k$ game: $(\cd_k, n)$, where loops and multiple edges are not counted when calculating the degree of each vertex; $(\cd'_k, n)$, where multiple edges are counted but loops are not; and $(\cd''_k, n)$, where all edges, including multiple edges and loops, are counted (every loop increases the degree of the vertex by two).
$(\cd'_k, n)$ will be useful in the next subsection, concerned with $k$-connectivity.

Recall the strategy $\mathcal S_{\min}$ presented in Subsection \ref{subsec:min_deg_process}. We show that a simple variant of this strategy, denoted $\mathcal{S}^\dagger_{\min}$, is optimal for the game $(\cd''_k, n)$. Utilizing Proposition \ref{pro:Gmin}, we conclude that $\mathcal{S}_{\min}$ is almost optimal for all of these three  games in some precise sense.
%Here, we show that the str%ategy of simulating the min-degree process $\{G_{\min}(n, m)\}_{m \geq 0}$ is nearly optimal for $(\cd_k, n)$ (in a way that is strong enough for our purposes), by showing that simulating the min-degree process $\{G''_{\min}(n, m)\}_{m \geq 0}$ with multiple edges and loops is optimal for $(\cd''_k, n)$.

Before stating our results, we need the following additional notation and terminology.
For two random variables $X$ and $Y$, taking values in $\mathbb{N}$, an integer $\ell \geq 0$, and a real number $\varepsilon \geq 0$, we say that $X$ \textit{$(\ell, \varepsilon)$-dominates} $Y$ if $\Pr(X \leq t+\ell) \geq \Pr(Y \leq t) - \varepsilon$ for any $t$. In our context, we identify each strategy $\mathcal{S}$ for a given game $\mathcal{G} = ({\mathcal P}, n)$ with its hitting time $H_\mathcal{G}(\mathcal{S})$ for this game, i.e., the random variable representing the number of rounds required for $\mathcal{S}$ to win $\mathcal{G}$. We say that $\mathcal{S}$ \textit{$(\ell, \varepsilon)$-dominates} another strategy $\mathcal{S}'$ if $H_\mathcal{G}(\mathcal{S})$ $(\ell, \varepsilon)$-dominates $H_\mathcal{G}(\mathcal{S}')$; in the special case $\ell = \varepsilon = 0$, we simply say that $\mathcal{S}$ \textit{dominates} $\mathcal{S}'$. $\mathcal{S}$ is \textit{$(\ell, \varepsilon)$-optimal} for a given game if it $(\ell, \varepsilon)$-dominates any other strategy $\mathcal{S}'$ for this game; if $\mathcal{S}$ is $(0, 0)$-optimal, we simply say that it is \textit{optimal}.

\begin{theorem} \label{th::MinDegkOnline}
	For every fixed positive integer $k$, the strategy $\mathcal S_{\min}$ is $(o(n), o(1))$-optimal for all three games $(\cd_k, n)$, $(\cd'_k, n)$, and $(\cd''_k, n)$.
\end{theorem}

Consider the following strategy, denoted $\mathcal S^\dagger_{\min}$, which is a slight variant of $\mathcal S_{\min}$. In any given round, let $G$ denote Builder's graph immediately before this round starts. Once a vertex $v$ is offered to Builder, we increase its degree by $1$, and only then choose a vertex $u$ u.a.r. among all vertices of minimum degree at this point. Hence, unlike in $\mathcal S_{\min}$, it is possible that $v$ was a vertex of minimum degree before it was offered but is not after we increase its degree by $1$, and so will surely not be chosen as the second vertex in this round. 

%For this purpose, the degree of $v$ is increased before making this choice, so even if $v$ was a minimum degree vertex before the round, we will not connect it to itself unless its degree is minimal \textit{after} the increase. 

The main technical ingredient in the proof of Theorem~\ref{th::MinDegkOnline} is the following lemma.
\begin{lemma} \label{lem:MinDegK''}
	$\mathcal S^\dagger_{\min}$ is optimal for $(\cd''_k, n)$. %, and dominates any strategy for $(\cd'_k, n)$ and $(\cd_k, n)$. 
\end{lemma}
%In particular, Lemma \ref{lem:MinDegK''} implies that all (deterministic or probabilistic) strategies that always connects the offered vertex to a minimum degree vertex are optimal.

The next lemma compares the performance of $\mathcal S_{\min}$ and $\mathcal S^{\dagger}_{\min}$, asserting that $\mathcal S_{\min}$ is indeed essentially optimal. 
\begin{lemma} \label{lem:S_min}
	$\mathcal{S}_{\min}$ $(o(n), o(1))$-dominates $\mathcal{S}^{\dagger}_{\min}$ for $(\cd''_k, n)$.
\end{lemma}

Theorem~\ref{th::MinDegkOnline} is an immediate corollary of Lemmas~\ref{lem:MinDegK''} and~\ref{lem:S_min}, and Propositions~\ref{pro:G''min} and~\ref{pro:Gmin}.
Before proceeding to the proofs of these lemmas, we briefly discuss previous results on the behavior of $\mathcal{S}_{\min}$.
Wormald showed, using his seminal differential equations method~\cite{Wormald, Wormald2}, that for every positive integer $k$ there exists a constant $h_k$ such that w.h.p. the min-degree process reaches minimum degree $k$ after $h_k n + o(n)$ rounds. This was explicitly shown in~\cite{Wormald, Wormald2} for $\{G'_{min}(n,m)\}_{m \geq 0}$ (corresponding to the game $(\cd'_k, n)$), but it is easy to show that it still holds for $\{G_{min}(n,m)\}_{m \geq 0}$ and $\{G''_{min}(n,m)\}_{m \geq 0}$ as well. Theorem~\ref{th::OnMinDegk} is thus an immediate corollary of Theorem~\ref{th::MinDegkOnline}. In fact, we obtain the following more general result as implied by Propositions~\ref{pro:Gmin} and~\ref{pro:G''min}.
\begin{corollary}
	Let $k$ be a positive integer. Then w.h.p. $\tau({\cd}_k, n) = (h_k + o(1)) n$. The same is true for $\tau(\cd'_k, n)$ and $\tau(\cd''_k, n)$.
\end{corollary}

The first few $h_k$'s were explicitly calculated in~\cite{KKRL}; it was shown there that
\begin{eqnarray*}
	&& h_1 = \ln 2 \approx 0.6931 \\
	&& h_2 = \ln 2 + \ln (1 + \ln 2) \approx 1.2197 \\
	&& h_3 = \ln ((\ln 2)^2 + 2(1 + \ln 2)(1 + \ln(1 + \ln 2))) \approx 1.7316
\end{eqnarray*}
Calculating $h_k$ for $k > 3$ can be carried out in a straightforward manner, by iteratively solving a simple differential equation with suitable initial conditions. For more details, see Subsections 3.1 and 3.2 of~\cite{Wormald2}.

We now proceed to the proofs of Lemmas~\ref{lem:MinDegK''} and~\ref{lem:S_min}.

\begin{proof}[Proof of Lemma~\ref{lem:MinDegK''}]
	For any integer $i \geq 0$, a strategy $\mathcal{S}$ of Builder is called \textit{$i$-minimizing} if in each of the first $i$ rounds, Builder chooses to connect the vertex he is offered to a vertex of minimum degree. $\mathcal{S}$ is said to be \textit{minimizing} if it is $i$-minimizing for every $i$. 
	In order to prove the lemma, it suffices to show that any $i$-minimizing strategy for $(\cd''_k, n)$ is dominated by some $(i+1)$-minimizing strategy. Indeed, seeing that domination is a transitive relation and that, trivially, any strategy is $0$-minimizing, the last statement implies that any strategy is dominated by a minimizing strategy. Moreover, any two minimizing strategies $\mathcal{S}$ and $\mathcal{S}'$ are clearly equivalent (in the sense that $\mathcal{S}$ dominates $\mathcal{S}'$ and $\mathcal{S}'$ dominates $\mathcal{S}$). Since $\mathcal S^{\dagger}_{\min}$ is a minimizing strategy, we conclude that it is optimal.
	
	Let $\mathcal{S}$ be an $i$-minimizing strategy, and consider the following $(i+1)$-minimizing strategy $\mathcal{S}'$;
	$\mathcal{S}'$ is identical to $\mathcal{S}$ in the first $i$ rounds.
	Conditioned on the degree sequence of Builder's graph immediately after the first $i$ rounds are completed and the vertex of round $i+1$, say $v_{i+1}$, is offered, for any vertex $v$ of the graph let $q_v$ denote the probability that, when playing according to $\mathcal{S}$, Builder chooses $v$ as the second vertex in round $i+1$.
	At this point, let $w$ be an arbitrary vertex of minimum degree and let $u$ be a vertex chosen randomly according to the distribution induced by $\mathcal{S}$, that is, for any vertex $v$, the probability that $u = v$ is $q_v$. In round $i+1$, when playing according to $\mathcal{S}'$, Builder claims an edge connecting $w$ and $v_{i+1}$.
	
	In the remainder of the game $\mathcal{S}'$ instructs Builder to play as follows.
	As long as $d_G(w) \leq d_G(u)$ (where $G$ denotes Builder's graph at any point during the game), $\mathcal{S}'$ imitates the behavior of $\mathcal{S}$ under the assumption that the second vertex chosen in round $i+1$ was $u$ (and not $w$, as was actually instructed by $\mathcal{S}'$).
	If, at some point, the degree of $w$ in Builder's graph exceeds that of $u$, then $\mathcal{S}'$ ``switches roles'' between these two vertices, that is, from now on, whenever $\mathcal{S}$ dictates that the offered vertex should be connected to $u$, $\mathcal{S}'$ dictates that it should be connected to $w$ instead, and vice versa. The behavior of $\mathcal{S}'$ with respect to other vertices is identical to that of $\mathcal{S}$.
	
	Clearly, $\mathcal{S}'$ is $(i+1)$-minimizing, and it is not hard to see that it dominates $\mathcal{S}$.
	Indeed, let $E$ denote the event that a role switch occurs (between some two vertices $u$ and $w$ as described) at some point when following the strategy $\mathcal{S}'$. Conditioning on $E$, the distribution of the hitting time of $\mathcal{S}$ is identical to that of $\mathcal{S}'$.
	On the other hand, conditioning on the complement of $E$, at any point during the game, and for any non-negative integer $\ell$, the probability that both vertices $u$ and $w$ have degree at least $\ell$ when playing according to $\mathcal{S}$ is at most the probability that both vertices have degree at least $\ell$ when playing according to $\mathcal{S}'$. In fact, if $u \neq w$, then after round $i+1$ the probability that $\min \{d_G(u), d_G(w)\} \geq \ell$ when playing according to $\mathcal{S}$ is equal to the probability that $\min \{d_G(u), d_G(w)\} \geq \ell+1$ when playing according to $\mathcal{S}'$. Hence, $\mathcal{S}'$ dominates $\mathcal{S}$ in this case as well.
\end{proof}

In the remainder of this subsection, a graph $H$ with vertex set $\{u_1, \ldots, u_t\}$ is said to be \textit{degree-dominated} by a graph $G$ with vertex set $\{v_1, \ldots, v_t\}$ if there exists a permutation $\pi : [t] \to [t]$ such that $d_H(u_i) \leq d_G(v_{\pi(i)})$ for every $1 \leq i \leq t$. 
For the proof of Lemma \ref{lem:S_min}, we will need the following fact, which can be straightforwardly proved by induction.
\begin{observation}
	\label{obs:degree_dominate}
	Suppose that $H$ and $G$ are graphs on the same number of vertices, such that $G$ degree-dominates $H$. Let $X_H$ (respectively, $X_G$) be the random variable representing the number of rounds required for Builder to reach minimum degree $k$ when following the strategy $\mathcal{S}_{\min}$, starting from the graph $H$ (respectively, $G$). Then $X_G$ dominates $X_H$. The same holds for $\mathcal S^{\dagger}_{\min}$.
\end{observation}

\begin{proof}[Proof of Lemma \ref{lem:S_min}]
	A round of the game played according to $\mathcal{S}_{\min}$ is considered to be a \textit{failure} if both of the following conditions are met.
	\begin{enumerate} [(a)]
		\item [$(a)$] The vertex $v$ offered in this round is of minimum degree.
		\item [$(b)$] $\mathcal S_{\min}$ instructs Builder to connect $v$ to itself in this round.
	\end{enumerate}
	We first show that the number of failures when playing according to $\mathcal{S}_{\min}$ is $o(n)$ w.h.p., and then we show how this implies the statement of the lemma.
	
	Let $N$ denote the number of vertices of minimum degree in Builder's graph immediately before a given round begins. The probability that both (a) and (b) above hold is $\frac{N}{n} \cdot \frac{1}{N} = \frac{1}{n}$. Since $\mathcal{S}_{\min}$ always reaches minimum degree $k$ after at most $kn$ rounds, the expected number of failures is bounded by $kn / n = k$, and thus, by Markov's inequality, the total number of failures is w.h.p. $o(n)$.
	
	Suppose now that we play a round of $\mathcal S_{\min}$ and of $\mathcal S^{\dagger}_{\min}$ in parallel, starting from the same graph $G$, and using the same source of randomness. Let $v$ be the vertex offered in this round. Observe that, conditioning on the event that this round is \textit{not} a failure for $\mathcal{S}_{\min}$, the distribution on the second vertex chosen in this round according to $\mathcal S_{\min}$ is identical to that of $\mathcal S^\dagger_{\min}$. On the other hand, suppose that this round is a failure, and one runs another round of $\mathcal S_{\min}$, increasing by one the degree of some vertex $u \neq v$ that was a minimum-degree vertex of $G$. The resulting graph, obtained by running two rounds of $\mathcal{S}_{\min}$ starting from $G$, degree-dominates any graph generated by one round of $\mathcal S^{\dagger}_{\min}$ starting from $G$. It thus follows by Observation~\ref{obs:degree_dominate} and the fact that w.h.p. the number of failures is $o(n)$, that $\mathcal{S}_{\min}$ $(o(n), o(1))$-dominates $\mathcal{S}^{\dagger}_{\min}$.
\end{proof}
%%%%
%%%%%%%%%%%%%%%%%%%%%%%%%%%%%%%%%%%%%%%%%%%%%%%%%%%%%%%%%%%%%%%%%%%%%%%%%%%%%%%
%%%%%%%%%%%%%%%%%%%%%%%%%%%%%%%%%%%%%%%%%%%%%%%%%%%%%%%%%%%%%%%%%%%%%%%%%%%%%%%

\subsection{Online $k$-connectivity}
%Recall the min-degree random multigraph process $\{G'_{\min}(n,m)\}_{m \geq 0}$ which was mentioned earlier. $G'_{\min}(n, 0)$ is the empty graph with vertex set $[n]$ and, for every $m \geq 0$, $G'_{\min}(n, m+1)$ is obtained from $G'_{\min}(n,m)$ by first choosing a vertex $u$ of minimum degree in $G'_{\min}(n,m)$ uniformly at random, and then connecting it by a new edge to a vertex $v \in [n] \setminus \{u\}$ chosen uniformly at random (note that this may result in multiple edges but not in loops). In a sense we are actually more interested in the graph process $\{G_{\min}(n,m)\}_{m \geq 0}$, but it will be easier to consider $\{G'_{\min}(n,m)\}_{m \geq 0}$ instead.
%

In this subsection we prove Theorem~\ref{thm:OnkCon}. Note that the lower bound is an immediate corollary of Theorem~\ref{th::OnMinDegk}. For the upper bound, we utilize a slightly modified min-degree process.
Consider a multigraph process $\{G^*_{\min}(n,m)\}_{m \geq 0}$, which is defined exactly like the process $\{G'_{\min}(n,m)\}_{m \geq 0}$ except that, instead of choosing the first vertex among all vertices of minimum degree, we choose it among all vertices with the smallest number of distinct neighbors. These two processes are identical as long as there are no multiple edges. Once there are multiple edges, the neighborhood of an endpoint of such edges could have minimum size while the degree of that endpoint could be strictly larger than the minimum degree. 
Consider the strategy $\mathcal{S}^*_{\min}$ defined as follows.

\textbf{Strategy $\mathcal S^{*}_{\min}$:}	Whenever Builder is offered some vertex $v$, he connects it to a vertex $u$, chosen u.a.r. among all vertices of $[n] \setminus \{v\}$ that have the smallest number of distinct neighbors.  

The following result is analogous to Proposition \ref{pro:G'min}, and its proof (which is omitted) is essentially the same as that of Proposition \ref{pro:G'min}. 
\begin{proposition}\label{pro:G*min}
	If $m = o \left(n^2 \right)$, then the strategy $\mathcal S^*_{min}$ is such that $H \sim G^*_{\min}(n,m)$ and $G \sim \mathcal S^*_{min}(n, (1+o(1))m)$ can be coupled
	in such a way that w.h.p.\ $H \subseteq G$.
\end{proposition}

For every positive integer $k$, let $H^*_k = H^*_k(n)$ denote the hitting time for the property that every vertex of $G^*_{\min}(n,m)$ has at least $k$ distinct neighbors, i.e.
$$
H^*_k = \min \{m : |N(u)| \geq k \textrm{ for every } u \in V(G^*_{min}(n,m))\}
$$
%$$
%H^*_k = \min \{m : \min |\{|N(u)| : u \in V(G^*_{\min}(n,m))\}| \geq k\}.
%$$
Furthermore, for every positive integer $k$, let $H_k$ denote the hitting time for the property that the minimum degree of $G_{\min}(n,m)$ is $k$, i.e.
$$
H_k = \min \{m : \delta(G_{\min}(n,m)) \geq k\}.
$$ 
We stress that $G_{\min}$ refers here to the min-degree process that does not allow multiple edges (as opposed to $G'_{\min}$) or loops. Recall the notion of $(\ell,\varepsilon)$-domination from Subsection~\ref{subsec:online_min_deg_k}.

\begin{lemma}\label{lem::G*}
Fix a positive integer $k$. Then $H^{*}_k$ $(\log n, o(1))$-dominates $H_k$.
\end{lemma}

\begin{remark}
	The $\log n$ term was chosen arbitrarily; it can be replaced with any function that tends to infinity with $n$.
\end{remark}

\begin{proof}
Consider a round of $G^{*}_{\min}$ to be a \textit{failure} if a multiple edge is chosen in this round.
For any multigraph $G$, let $simp(G)$ denote the simple graph $H$ so that $uv$ is an edge of $H$ if and only if $uv$ appears at least once in $G$.
It suffices to prove the following two statements. 

\begin{enumerate}
	\item For any multigraph $G$, the following two edge distributions are identical.
	\begin{enumerate}
		\item The edge distribution of a single round of $G_{\min}$ starting from $simp(G)$.
		\item The edge distribution of a single round of $G^*_{\min}$ starting from $G$, conditioned on the event that this round is not a failure.
	\end{enumerate}
	\item For any fixed $k$, the number of failures of $G^{*}_{\min}$ until the point that any vertex in the generated graph has at least $k$ distinct neighbors  is w.h.p.\ at most $\log n$.
\end{enumerate} 
We start by proving the first statement. Note that a vertex $v$ has minimum degree in $simp(G)$ if and only if it has a minimum number of distinct neighbors in $G$; let $V_{\min}$ denote the set of all such vertices. Let $N = |V_{\min}|$ and let $\delta$ denote the degree of the vertices of $V_{\min}$ in $simp(G)$. The probability of an edge $uv$ to be chosen according to each of the distributions is $|\{u, v\} \cap  V_{\min}| / N (n-1-\delta)$, so the distributions are indeed identical.

To prove the second statement, observe that the probability for a multiple edge to be chosen in a single round of $G^*_{\min}$ is bounded from above by $k / (n-1)$, and thus the total expected number of failure rounds in $G^{*}_{\min}(n, m)$ is $O(k^2)$ as long as $m = O(kn)$. Putting, say, $m = nk+\log{n}$, we deduce by Markov's inequality that w.h.p.\ (for fixed $k$) the total number of failures in $G^{*}_{\min}(n, m)$ is bounded from above by $\log n$. Conditioning on this event, we know that the generated graph had at least $m - \log{n} = kn$ successful rounds among its first $m$ rounds. Hence, it must already hold that any vertex has at least $k$ distinct neighbors at this point, concluding the proof.
\end{proof}

Building on Proposition~\ref{pro:G*min} and Lemma~\ref{lem::G*}, Theorem~\ref{thm:OnkCon} is now an immediate corollary of the following result which strengthens a result from~\cite{KKRL}; our proof builds on their method.

\begin{theorem} \label{th::k-con}
	Let $k \geq 3$ be a fixed integer. Then w.h.p. $G^*_{\min}(n, \alpha n)$ is $k$-connected if $\alpha > h_k$ and is not $k$-connected if $\alpha < h_k$.
\end{theorem}

\begin{remark}
	It follows from the results in~\cite{KKRL} that $k$ cannot be taken to be smaller than $3$ in Theorem~\ref{th::k-con}.
\end{remark}

In the proof of Theorem~\ref{th::k-con} we will make use of the following auxiliary lemma.
\begin{lemma} \label{lem::smallSets}
	Let $k$ be a positive integer and let $G \sim G^*_{\min}(n, \alpha n)$, where $\alpha \leq k$. Then w.h.p. $e_G(A) \leq |A|$ for every set $A \subseteq [n]$ of size $1 \leq |A| \leq 101 k^2$.
\end{lemma}

\begin{proof}
	Fix some integer $k \geq 1$ and let $M = 200 k^2$. For every $1 \leq t \leq M$, a round of the process is said to be of \emph{type $t$} if at the start of that round, the number of vertices whose neighborhood is of minimum size is larger than $n^{(t-1)/M}$ and is at most $n^{t/M}$. Since $\alpha \leq k$, it follows that $\delta(G) \leq 2k$ holds throughout the process. Moreover, it follows by the description of the process that in every round we increase the number of neighbors of some vertex whose neighborhood is of minimum size or we choose a multiple edge. Since, by the proof of Lemma~\ref{lem::G*}, w.h.p. there are at most $\log n$ rounds in which we choose a multiple edge, it follows that w.h.p., throughout the process there are at most $2k n^{t/M} + \log n \leq 3k n^{t/M}$ rounds of type $t$ for every $1 \leq t \leq M$.
	
	Fix an integer $1 \leq i \leq 101 k^2$ and a set $A \subseteq [n]$ of size $i$. We would like to bound from above the probability that $e_G(A) \geq i+1$. Let $S_i = \left\{(s_1, \ldots, s_M) \in {\mathbb N}^M : \sum_{t=1}^M s_t = i+1 \right\}$. For every $\bar{s} = (s_1, \ldots, s_M) \in S_i$, let $p_{i,\bar{s}}$ denote the probability that, for every $1 \leq t \leq M$, at least $s_t$ edges with both endpoints in $A$ were claimed during rounds of type $t$. Then
	\begin{eqnarray*}
		Pr(e_G(A) \geq i+1) &\leq& \sum_{\bar{s} \in S_i} p_{i, \bar{s}} \\
		&\leq& \sum_{(s_1, \ldots, s_M) \in S_i} \prod_{t=1}^M \binom{3k n^{t/M}}{s_t} \left(\frac{i}{n^{(t-1)/M}} \right)^{s_t} \left(\frac{i}{n} \right)^{s_t}\\
		&\leq& \sum_{(s_1, \ldots, s_M) \in S_i} c_k n^{\sum_{t=1}^M (t s_t/M - (t-1) s_t/M - s_t)} = c'_k n^{(i+1)(1/M - 1)},
	\end{eqnarray*}
	where $c_k$ and $c'_k$ are appropriate constants, depending on $k$ but not on $n$.
	
	A union bound over all relevant choices of $A$ then shows that the probability that there exists a set $A \subseteq [n]$ such that $|A| = i$ for some $1 \leq i \leq 101 k^2$ and $e_G(A) \geq i+1$ is at most
	$$
	\sum_{i=1}^{101 k^2} \binom{n}{i} c'_k n^{(i+1)(1/M - 1)} \leq \sum_{i=1}^{101 k^2} c'_k n^{i + (i+1)(1/M - 1)} \leq c'_k\cdot 101 k^2 \cdot n^{-1} \cdot n^{(101 k^2 + 1)/(200 k^2)} = o(1),
	$$
	where the last inequality holds since $M = 200 k^2$ and $i \leq 101 k^2$.
\end{proof}

We are now in a position to prove Theorem~\ref{th::k-con}.

\begin{proof}
	If $G$ is $k$-connected, then, in particular, every vertex of $G$ has at least $k$ distinct neighbors. It follows by the definitions of $H^*_k$ and $h_k$ that w.h.p. $G^*_{\min}(n, \alpha n)$ is not $k$-connected if $\alpha < h_k$. Assume then that $\alpha > h_k$. Since $k$-connectivity is a monotone increasing property and $h_k < k$, we can assume that $\alpha \leq k$. In order to prove that $G \sim G^*_{\min}(n, \alpha n)$ is w.h.p. $k$-connected, we will show that the probability that there exist pairwise disjoint sets $S, R$, and $T$ such that $[n] = S \cup R \cup T$, $|R| = k-1$, $1 \leq |S| \leq |T|$, and $E_G(S,T) = \emptyset$, tends to $0$ as $n$ tends to infinity. Since, by the definitions of $H^*_k$ and $h_k$, w.h.p. every vertex of $G$ has at least $k$ distinct neighbors, we can restrict our attention to the case $|S| \geq 2$.
	
	Fix a triple $S, R, T$ as above, where $|S| = s$ for some $2 \leq s \leq 100 k^2$. Let $A = S \cup R$ and observe that $|A| = s + k - 1 \leq 101 k^2$. It follows by Lemma~\ref{lem::smallSets} that w.h.p. $e_G(A) \leq s + k - 1$ and $e_G(S) \leq s$. Since, moreover, w.h.p. every vertex of $G$ has at least $k$ distinct neighbors, if $E_G(S,T) = \emptyset$, then $e_G(A) \geq k s - s$. Since $k \geq 3$ and $s \geq 2$, this is a contradiction unless $k = 3$ and $s = 2$. In the latter case $|A| = 4$ and $e_G(A) \geq 5$ which again contradicts Lemma~\ref{lem::smallSets}.
	
	Now, fix a triple $S, R, T$ as above, where $100 k^2 \leq |S| \leq (n-k+1)/2$. A round of the process is said to be \emph{bad} if, in that round, the first vertex is chosen from $R$, \emph{good} if it is chosen from $T$, and \emph{great} if it is chosen from $S$. It suffices to prove that the probability that no edges between $S$ and $T$ were claimed in any round which is not bad is $o(1)$. Since $\alpha \leq k$, it follows that $\delta(G) \leq 2k$ holds throughout the process. Since, moreover, the first vertex chosen in every round has the least number of distinct neighbors and there are at most two edges between any pair of vertices by Lemma~\ref{lem::smallSets}, there can be at most $2 \cdot 2 k |R| \leq 4k^2$ bad rounds. Let $X_S$ be the random variable which counts the number of great rounds. Since $\delta(G) \geq k$ holds w.h.p. at the end of the process, and there are at most $4k^2$ bad rounds, if $E_G(S,T) = \emptyset$, then $X_S \geq k |S|/2 - 2k^2$. Therefore, the probability that $S, R, T$ as above exist is at most
	
	\begin{eqnarray} \label{eq::largeSets}
	&& \sum_{s = 100 k^2}^{(n-k+1)/2} \sum_{i = k s/2 - 2 k^2}^{\alpha n} \binom{n}{s} \binom{n}{k-1} Pr(X_S = i) Pr(E_G(S,T) = \emptyset \; | \; X_S = i) \nonumber \\
	&\leq& n^{k-1} \sum_{s = 100 k^2}^{(n-k+1)/2} \sum_{i = k s/2 - 2k^2}^{\alpha n} \binom{n}{s} \left(\frac{s+k-2}{n-1} \right)^i \left(\frac{n-s-1}{n-1} \right)^{\alpha n - i - 4 k^2}
	\end{eqnarray}
	
	It follows from Stirling's formula that $\binom{n}{s} \leq \frac{n^n}{s^s (n-s)^{(n-s)}}$ for every $n$ and $s$. Hence, a straightforward calculation shows that
	$$
	\binom{n}{s} \left(\frac{s+k-2}{n-1} \right)^s \left(\frac{n-s-1}{n-1} \right)^{\alpha n - s - 4 k^2} \leq e^k \left(1 - \frac{s}{n}\right)^{(\alpha - 1) n - 4 k^2} < e^{-s/2},
	$$
	where the last inequality holds since $h_k \geq h_3 > 1.7$ holds for every $k \geq 3$. Therefore~\eqref{eq::largeSets} can be bounded from above by
	\begin{eqnarray*}
		&& n^{k-1} \sum_{s = 100 k^2}^{(n-k+1)/2} \sum_{i = k s/2 - 2k^2 - s}^{\alpha n - s} \left(\frac{s+k-2}{n-s-1} \right)^i e^{-s/2} \\
		&\leq& n^{k-1} \sum_{s = 100 k^2}^{\log^2 n} \sum_{i = k s/2 - 2k^2 - s}^{\alpha n - s} \left(\frac{s+k-2}{n-s-1} \right)^i + \alpha n^k \sum_{s = \log^2 n}^{(n-k+1)/2} e^{-s/2} \\
		&\leq& \alpha n^k \log^2 n \left(\frac{2 \log^2 n}{n} \right)^{48 k^2} + \alpha n^{k+1} e^{- \log^2 n/2} = o(1),
	\end{eqnarray*}
	where the last inequality holds for every $k \geq 3$.
\end{proof}

%%%%%%%%%%%%%%%%%%%%%%%%%%%%%%%%%%%%%%%%%%%%%%%%%%%%%%%%%%%%%%%%%%%%%%%%%%%%%%%

\section{Concluding remarks and open problems}\label{sec::openprob}

In this paper we have initiated the research on the semi-random graph process, leading to many intriguing open questions. We mention just a few of the possible directions for future research.

\medskip

\noindent \textbf{Online fixed graph.} Let $H$ be an arbitrary fixed graph and let $d$ be the degeneracy of $H$. It is proved in Theorem~\ref{th::fixedGraphUpperBound} that w.h.p. $\tau({\mathcal P}_H, n) = O(n^{(d-1)/d})$. For the special case $H = K_{d+1}$, a lower bound of the same order of magnitude is proved in Theorem~\ref{thm:lowerCliqueOnLine}. We believe that such a bound holds for any graph $H$.

\begin{conjecture} \label{conj::fixedGraph}
	Let $H$ be an arbitrary fixed graph and let $d$ be the degeneracy of $H$. Then w.h.p. $\tau({\mathcal P}_H, n) = \Theta(n^{(d-1)/d})$.
\end{conjecture}

Note that the assertion of Conjecture~\ref{conj::fixedGraph} is trivially true for $d=1$, that is, when $H$ is a forest.

\medskip

\noindent \textbf{Perfect matching.} Recall that $\mathcal{PM}$ denotes the property of containing a perfect matching. It follows from our results that w.h.p.
\begin{equation} \label{eq::PM}
(\ln 2 + o(1)) n = \tau'(\mathcal{PM}, n) \leq \tau(\mathcal{PM}, n) \leq (1 + 2/e + o(1)) n,
\end{equation}
where the last inequality holds by Corollary~\ref{cor:KoutBip}. On the other hand, the equality is a simple corollary of Proposition~\ref{prop:FindingH}. Indeed,  let $H$ be a matching consisting of $n/2$ edges (for convenience, we will assume that $n$ is even). Observe that in every orientation of $H$ there are precisely $n/2$ vertices of out-degree $1$ and precisely $n/2$ vertices of out-degree $0$. Therefore, using the notation of Subsection \ref{sec::BallsBins}, we have $\tau'(\mathcal{PM}, n) = m(H) = \min \{m : X_0^m \leq n/2\}$. The required equality now follows by~\eqref{e:exactly-k} and since $X_0^m$ is concentrated around its mean by Lemma~\ref{lem::Balls}. We believe that neither the lower nor the upper bound in~\eqref{eq::PM} is tight. It would be interesting to close or at least reduce the gap between these two bounds.

\medskip

\noindent \textbf{Hamilton cycle.} Recall that $\ch$ denotes the property of admitting a Hamilton cycle. Similarly to the case of a perfect matching, it is not hard to show that w.h.p.
\begin{equation} \label{eq::HAMoff}
\tau'(\ch, n) = (\alpha + o(1)) n,
\end{equation}
where $\alpha = 1.14619...$ is the unique positive real number satisfying $1 = (2 + \alpha) e^{-\alpha}$ (it is straightforward to verify that there exists a unique positive real number which satisfies this equation). The equality~\eqref{eq::HAMoff} is a simple corollary of Proposition~\ref{prop:FindingH}. Indeed, let $H$ be a cycle of length $n$. Observe that in every orientation of $H$, there are $r$ vertices of out-degree $1$, $(n-r)/2$ vertices of out-degree $2$, and $(n-r)/2$ vertices of out-degree $0$, for some integer $0 \leq r \leq n$. Therefore, a necessary and sufficient condition for Builder to construct a Hamilton cycle is $\sum_{k=2}^n X_k^m \geq (n - X_1^m)/2$ which is equivalent to $n - 2 X_0^m - X_1^m \geq 0$. Setting $m = (c + o(1)) n$ and using~\eqref{e:exactly-k} and Lemma~\ref{lem::Balls}, we conclude that w.h.p. the aforementioned necessary and sufficient condition holds for $c$ which satisfies $0 = (1 - 2 e^{-c} - c e^{-c}) n$ as required.

For the online Hamilton cycle game we have the following bounds (which hold w.h.p.):
\begin{equation} \label{eq::HAMon}
(h_2 - o(1)) n \leq \tau(\ch, n) \leq (3 + o(1)) n,
\end{equation}
where $h_2 = \log 2 + \log (1 + \log 2) \approx 1.219736$ (as explicitly calculated in \cite{Wormald2}). Indeed, the lower bound holds by Theorem~\ref{th::OnMinDegk} since minimum degree at least $2$ is a trivial necessary condition for Hamiltonicity. On the other hand, the upper bound holds by Corollary~\ref{cor:Kout} and the well known result asserting that a random graph generated by the $3$-out model is w.h.p.\@ Hamiltonian~\cite{bohman2009hamilton}. It would be interesting to close or at least reduce the gap between the lower and upper bounds in~\eqref{eq::HAMon}.

\medskip

\noindent \textbf{Bounded degree graphs.} Let $H$ be a graph with vertex set $[n]$ and with bounded maximum degree. Observe that $\tau({\mathcal P}_H, n) \leq (1 + o(1)) n \log n$. Indeed, Builder can construct $H$ as follows. For every $1 \leq i \leq n$, let $j_1, \ldots, j_{d_i}$ be an arbitrary ordering of the neighbors of $i$ in $H$. In each round, if Builder is offered some vertex $i$ for the $r$th time for some $1 \leq r \leq d_i$, then Builder claims the edge $i j_r$, otherwise he claims an arbitrary edge. Using this strategy, it is evident that $\tau({\mathcal P}_H, n) \leq m$, where $m$ is the smallest integer such that, during the first $m$ rounds, every $1 \leq i \leq n$ is offered at least $\Delta(H)$ times. Using~\eqref{e:exactly-k} and Lemma~\ref{lem::Balls}, a straightforward calculation shows that $m = (1 + o(1)) n \log n$.

Note that $\tau({\mathcal P}, n) = O(n)$ holds for every property ${\mathcal P}$ we considered in this paper. This observation has led Noga Alon to ask us the following question.
\begin{question} \label{q::linearUpperBound}
	Is there a graph $H$ on $n$ vertices with bounded maximum degree such that $\tau({\mathcal P}_H, n) = \omega(n)$?
\end{question}

\medskip

Another possible direction of future research is the study of natural variations of our process. These include the following:

\medskip

\noindent \textbf{A digraph process.} Consider the same semi-random graph process, except that whenever Builder claims an edge, he must orient it from the vertex he was offered to the vertex he chose. His goal now is to build a digraph which satisfies some predetermined increasing property as soon as possible. Consider for example the aim of building a directed Hamilton cycle. Let $H$ be an undirected cycle on $n$ vertices and let $D_1$ and $D_2$ be orientations of $H$ such that in $D_1$ the out-degree of every vertex is $1$ and in $D_2$ the out-degree of every vertex is either $0$ or $2$ (assume for convenience that $n$ is even). It is not hard to see that $\tau'({\mathcal P}_{D_1}, n) = \tau({\mathcal P}_{D_1}, n) = (1 + o(1)) n \log n$. Both equalities follow from the fact that Builder can construct $D_1$ as soon as every vertex is offered at least once but not sooner. Indeed, if some vertex is never offered, then its out-degree in Builder's graph will be $0$ (both in the offline and in the online games). On the other hand, in the online game, Builder can play as follows: for every $1 \leq i \leq n$, the first time vertex $i$ is offered, Builder connects it to vertex $(i \mod n) + 1$; in any other round he plays arbitrarily. This proves the first equality. Using~\eqref{e:exactly-k} and Lemma~\ref{lem::Balls}, a straightforward calculation proves the second equality. In light of~\eqref{eq::HAMoff} and~\eqref{eq::HAMon}, this shows that, in general, the digraph process behaves very differently than the graph process. Now, consider constructing $D_2$. Similarly to the case of an undirected Hamilton cycle, one can show that $\tau'({\mathcal P}_{D_2}, n) = \Theta(n)$. This shows that (at least in the offline case) the digraph process in which Builder aims to build some digraph $D$ might really depend on $D$ and not just on its underlying undirected graph. As for the online case, the following question seems plausible.
\begin{question} \label{q::DirectedHamCycle}
	Is it true that for every $\varepsilon > 0$ there exist constants $C$ and $n_0$ such that $\tau({\mathcal P}_D, n) \leq C n$ holds for every $n \geq n_0$ and every orientation $D$ of the $n$-cycle $H$ in which the number of vertices of out-degree $0$ is at least $\varepsilon n$?
\end{question}

\medskip

\noindent \textbf{Non-uniform sampling.} In the process we studied in this paper, the vertex Builder was offered in every round was chosen u.a.r. One could also study a similar process where the vertices Builder is offered are chosen according to some other probability distribution (which can differ between rounds). For example, consider the following random process which was studied in~\cite{RW1, RW2, RW3}. For a positive integer $d$, the random $d$-process $\{G_i\}_{i=0}^N$, where $N = \lfloor n d/2 \rfloor$, is defined as follows. $G_0$ is the empty graph on $n$ vertices and, for every $i \geq 0$, $G_{i+1} = G_i \cup e_{i+1}$, where $e_{i+1} = uv$ is chosen u.a.r.~from the set of all non-edges of $G_i$ for which $\max \{d_{G_i}(u), d_{G_i}(v)\} < d$. While we cannot use our process as is to approximate the random $d$-process, we can easily do so if the vertices Builder is offered in every round are chosen u.a.r.~from the set of all vertices whose degree is strictly smaller than $d$. Another example of a random graph process we can approximate by offering Builder vertices according to an appropriate probability distribution is the min-min random graph process~\cite{CK}. One can of course consider various probability distributions for the aforementioned digraph process as well.

\medskip

\noindent \textbf{Delaying increasing graph properties.} In this paper, we studied $\tau({\mathcal P}, n)$ (and, similarly, $\tau'({\mathcal P}, n)$ for the offline game) which is the \textbf{smallest} number of rounds in the online game Builder needs in order to build a graph on $n$ vertices which satisfies the increasing graph property ${\mathcal P}$. Instead, we can have Builder try to avoid satisfying ${\mathcal P}$ for as long as possible. Formally, we define $T({\mathcal P}, n)$ (and, similarly, $T'({\mathcal P}, n)$ for the offline game) to be the \textbf{largest} number of rounds in the online game for which Builder can maintain a graph on $n$ vertices which does not satisfy the increasing graph property ${\mathcal P}$. Note that in order for this to make sense, we can no longer allow Builder to create loops or multiple edges. Consider for example the property ${\mathcal P}_t$ of containing a connected  component on at least $t$ vertices. It is trivial that $\tau({\mathcal P}_t, n) = \tau'({\mathcal P}_t, n) = t-1$. On the other hand, studying $T({\mathcal P}_t, n)$ (and to some extent also $T'({\mathcal P}_t, n)$) seems to have merit, especially with relation to the phase transition in the size of the largest component. Another interesting example is the property ${\mathcal P}_{\Delta}$ of being triangle-free. The problem of determining $T({\mathcal P}_{\Delta}, n)$ and $T'({\mathcal P}_{\Delta}, n)$ is related to classical problems in extremal graph theory and to other restricted random graph processes (see, e.g.,~\cite{Bohman, BK, FGM}).

%%%%%%%%%%%%%%%%%%%%%%%%%%%%%%%%%%%%%%%%%%%%%%%%%%%%%%%%%%%%%%%%%%%%%%%%%%%%%%%
%%%%%%%%%%%%%%%%%%%%%%%%%%%%%%%%%%%%%%%%%%%%%%%%%%%%%%%%%%%%%%%%%%%%%%%%%%%%%%%

\section*{Acknowledgements}
We would like to thank the anonymous referees for helpful comments. We would like to thank Peleg Michaeli for suggesting the study of this model and to thank Noga Alon and Michael Krivelevich for helpful discussions. The research on this project was initiated during a joint research workshop of Tel Aviv University and the Free University of Berlin on Positional Games and Extremal Combinatorics, held in Berlin in 2016; we would like to thank both institutions for their support.
%
%\noindent\textbf{Acknowledgement}
%	The research on this project was initiated during a joint
%	research workshop of Tel Aviv University and the Free University
%	of Berlin on
%Positional Games,
%	held in Berlin in February 2016. We
%	would like to thank both institutions for their support.
%	The authors would also like to Peleg Michaeli for offering the problem for Hamilton cycles and Michael Krivelevich for helpful discussions.
%


\begin{thebibliography}{99}

	\bibitem{AloSpe2008}
	N.\ Alon and J. H. Spencer, {\bf The Probabilistic Method}, $4^{th}$ ed. Wiley,
	New York, 2016.

\bibitem{Azuma}
K. Azuma, Weighted sums of certain dependent variables. \emph{Tohoku Math. J.} 3 (1967): 357--367.

\bibitem{Bohman}
T. Bohman, The Triangle-Free Process. \emph{Advances in Mathematics}, 221 (2009): 1653--1677.

\bibitem{bohman2009hamilton}
T. Bohman and A. Frieze, Hamilton cycles in 3-out. \textit{Random Structures and Algorithms} 35, no. 4 (2009): 393--417.

\bibitem{BK}
T. Bohman and P. Keevash, Dynamic concentration of the triangle-free process. \textit{arXiv preprint} arXiv:1302.5963.

\bibitem{CK}
A. Coja-Oghlan and M. Kang, The evolution of the min–min random graph process. \emph{Discrete Mathematics} 309 (2009): 4527--4544.


\bibitem{ER}
P. Erd\H{o}s and A. R\'enyi, On random graphs, I. \textit{Publicationes Mathematicae (Debrecen) } 6 (1959): 290--297.

\bibitem{FGM}
G. Fiz Pontiveros, S. Griffiths, and R. Morris, The triangle-free process and the Ramsey number $R(3,k)$.
\emph{Mem. Amer. Math. Soc.}, to appear.

\bibitem{frieze1986maximum}
A. M. Frieze, Maximum matchings in a class of random graphs. \textit{Journal of Combinatorial Theory, Series B} 40, no. 2 (1986): 196--212.


\bibitem{HKK}
P. Haxell, M. Krivelevich, and G. Kronenberg, Goldberg's Conjecture is True for Random Multigraphs. \textit{arXiv preprint} arXiv:1803.00908.

\bibitem{JLR}
S. Janson, T. \L uczak and  A. Ruci\'nski, {\bf Random graphs}, Wiley, 2000.

\bibitem{KKRL}
M. Kang, Y. Koh, S. Ree and T. \L uczak, The connectivity threshold for the min-degree random graph process. \emph{Random Structures and Algorithms} 29 (2006): 105--120.


\bibitem{KS}
M. Kang and T. G. Seierstad, Phase transition of the minimum degree random multigraph process. \emph{Random Structures and Algorithms} 31 (2007): 330--353.

\bibitem{KOP}
M. Karo\'nski, E. Overman and B. Pittel, On a perfect matching in a random bipartite digraph with average out-degree below two. \textit{arXiv preprint} arXiv:1903.05764.

%\bibitem{KP}
%M. Karo\'nski and B. Pittel, Existence of a perfect matching in a random $(1 + e^{-1})-out$ bipartite graph. \emph{Journal of Combinatorial Theory Ser. B} (2003): 1--16.

\bibitem{Scottish}
R. D. Mauldin, \textbf{The Scottish Book: Mathematics from the Scottish Cafe}, Birkh\"auser
Boston, 1981.

\bibitem{MR}
M. Molloy and B. Reed, \textbf{Graph Colouring and the Probabilistic Method}, Springer, 2002.

\bibitem{RS}
M. Raab and A. Steger, Balls into Bins - A simple and tight analysis. 
\textit{in: International Workshop on Randomization and Approximation Techniques in Computer Science}, Springer, Berlin, Heidelberg, (1998): 159--170.

\bibitem{RW1}
A. Ruci\'nski and N. C. Wormald, Random graph processes with degree restrictions. \emph{Combinatorics, Probability and Computing} 1 (1992): 169--180.

\bibitem{RW2}
A. Ruci\'nski and N. C. Wormald, Random graph processes with maximum degree 2. \emph{Annals of Applied Probability} 7 (1997): 183--199.

\bibitem{RW3}
A. Ruci\'nski and N. C. Wormald, Connectedness of graphs generated by a random $d$-process. \emph{J. Austral. Math. Soc.} 72 (2002): 67--85.

\bibitem{Talagrand}
M. Talagrand, Concentration of measure and isoperimetric inequalities in product spaces. \emph{Inst.~Hautes \'Etudes Sci.~Publ.~Math.} 81 (1995): 73--205.

\bibitem{walkup1980matchings}
D. W. Walkup,  Matchings in random regular bipartite digraphs. \textit{Discrete Mathematics} 31, no. 1 (1980): 59--64.

\bibitem{Wormald}
N. C. Wormald, Differential equations for random processes and random graphs. \emph{Annals of Applied Probability} 5 (1995): 1217--1235.

\bibitem{Wormald2}
N. C. Wormald, The differential equation method for random graph
processes and greedy algorithms.  \emph{Lectures on Approximation and Randomized Algorithms, M. Karo\'nski and H. J. Pr\"omel (eds)},  PWN, Warsaw (1999): 73--155.
\end{thebibliography}
\end{document}